\documentclass[3p]{elsarticle}

\usepackage{amssymb,amsthm}                 

\usepackage{amssymb,bm,url}  
\usepackage{amsmath}
\usepackage{enumerate,graphicx}
\usepackage[mathscr]{eucal}
\usepackage{verbatim,url}
\usepackage{tikz}
\usetikzlibrary{matrix,arrows,decorations}
\usepackage{calc}
\usepackage{txfonts}
\usepackage{geometry}

\newenvironment{newlist}
   {\begin{list}{}{\setlength{\labelsep}{0.25cm}
                   \setlength{\labelwidth}{0.65cm}
                      \setlength{\leftmargin}{0.9cm}}}
   {\end{list}}

\newenvironment{claimslist}
   {\begin{list}{}{\setlength{\labelsep}{0.25cm}
   			  \setlength{\labelwidth}{2.3cm}
                   \setlength{\leftmargin}{1.53cm}}}
   {\end{list}}

\newenvironment{caseslist}
   {\begin{list}{}{\setlength{\labelsep}{0.25cm}
                   \setlength{\labelwidth}{2.1cm}
                   \setlength{\leftmargin}{1.35cm}}}
   {\end{list}}

\renewcommand{\le}{\leqslant}
\renewcommand{\ge}{\geqslant}
\renewcommand{\leq}{\leqslant}
\renewcommand{\geq}{\geqslant} 
\renewcommand{\phi}{\varphi}  

\newcommand{\defn}[1]{{\emph{#1}}}
\newcommand{\cat}[1]{\boldsymbol{\mathscr{#1}}} 
\newcommand{\str}[1]{\mathbf{#1}} 
\newcommand{\alg}[1]{\str{#1}} 
\newcommand{\spc}[1]{\boldsymbol{#1}} 
\newcommand{\fnt}[1]{\mathsf{#1}} 
\newcommand{\ope}[1]{\mathbb{#1}} 

\newcommand{\CA}{{\cat A}} 
\newcommand{\CB}{\cat B} 
\newcommand{\CCD}{\cat D} 
\newcommand{\CP}{\cat P} 
\newcommand{\CV}{\cat V} 
\newcommand{\CQ}{{\cat Q}} 
\newcommand{\CX}{\cat X} 
\newcommand{\DB}{\cat{DB}} 

\newcommand{\D}{\fnt D} 
\newcommand{\E}{\fnt E} 
\newcommand{\ED}{\fnt{ED}} 
\newcommand{\Free}{\fnt{F}} 

\newcommand{\Tp}{\mathscr{T}}

\newcommand{\A}{\alg A}
\newcommand{\B}{\alg B}

\newcommand{\Lalg}{\alg L}

\newcommand{\K}{\alg K}
\newcommand{\M}{\alg M}
\renewcommand{\S}{\alg S}
\newcommand{\X}{\spc X}
\newcommand{\Y}{\spc Y}
\newcommand{\Z}{\spc Z}
\newcommand{\two}{\boldsymbol 2}

\newcommand{\four}{\boldsymbol 4}
\newcommand{\SEVEN}{\mathcal{SEVEN}}
\newcommand{\FOUR}{\mathcal{FOUR}}
\newcommand{\CM}{\cat M}
\newcommand{\MT}{\twiddle{\spc{M}}}
\newcommand{\CMT}{\twiddle{\cat M}}

\newcommand{\mbf}{\mathbf{f}}
\newcommand{\mbt}{\mathbf{t}}
\newcommand{\mbdf}{\mathbf{df}}
\newcommand{\mbdt}{\mathbf{dt}}
\newcommand{\mbdT}{\mathbf{d}\boldsymbol\top}

\newcommand{\Seq}{\mathfrak{S}}

\newcommand{\twiddle}[1]{{\smash{\underset{\raise.375ex\hbox{$\smash\sim$}}
       {#1}}\vphantom{\underline{#1}}}} 
\newcommand{\twoT}{\twiddle{ 2}}

\DeclareMathOperator{\Con}{Con}

\newcommand{\du}{\smash{\,\cup\kern-0.45em\raisebox{1ex}{$\cdot$}}\,\,}
\newcommand{\dubig}{\smash{\bigcup\kern-0.74em\raisebox{1.2ex}{$\cdot$}}\ }

\newcommand{\conv}[1]{\smash{\overset{\hbox{\lower1.8ex\hbox{$\smash{\scriptstyle{\smallsmile}}$}}}{#1}}}

 \DeclareMathOperator{\graph}{graph}
 \DeclareMathOperator{\dom}{dom}
 \DeclareMathOperator{\ISP}{\ope{ISP}}
 \DeclareMathOperator{\HSP}{\ope{HSP}}
 \DeclareMathOperator{\IScP}{{\ope{IS} _{\mathrm{c}}
 \ope{P}^+}}
\DeclareMathOperator{\Id}{id}
\newcommand{\restrict}{\,{\mathbin{\vert\mkern-0.3mu\grave{}}}\,}

\newcommand{\lincol}{black}
\newcommand{\linth}{thick}

\newcommand{\epo}[1]{\filldraw[fill=white,\linth](#1) circle (2pt);}
\newcommand{\li}[1]{\draw[\linth,\lincol] #1;}
\newcommand{\dotli}[1]{\draw[dotted,\linth,\lincol] #1;}
\newcommand{\circsize}{3}
\newcommand{\sqsize}{0.1}
\newcommand{\pdonel}[1]{\filldraw[fill=white,\linth](#1) circle (\circsize pt);}
\newcommand{\pdoner}[1]{\filldraw[fill=black](#1) circle (\circsize pt);}
\newcommand{\pdtwol}[1]{\filldraw[fill=white,\linth](#1) +(-\sqsize,-\sqsize) rectangle ++(\sqsize,\sqsize);}
\newcommand{\pdtwor}[1]{\filldraw[fill=black](#1) +(-\sqsize,-\sqsize) rectangle ++(\sqsize,\sqsize);}

\newtheorem{thm}{Theorem}[section]
\newtheorem{lem}[thm]{Lemma}
\newtheorem{coro}[thm]{Corollary}
\newtheorem{prop}[thm]{Proposition}
\theoremstyle{definition}
\newtheorem{ex}[thm]{Example}

\begin{document}
\begin{frontmatter}

\title{Product representation for default bilattices: \\ 
an application of natural duality theory}

\author[lmc]{L.\,M.\,Cabrer}
\ead{l.cabrer@disia.unifi.it}
\address[lmc]{Dipartimento di Statistica, Informatica, Applicazioni, Universit\`a degli Studi di Firenze, 59 Viale Morgani, 50134, Italy}

\author[apkc]{A.\,P.\,K. Craig\corref{cor1}}
\ead{acraig@uj.ac.za}
\address[apkc]{Department of Mathematics, University of Johannesburg, PO Box 524, Auckland Park, 2006, South Africa}

\author[hap]{H.\,A. Priestley}
\ead{hap@maths.ox.ac.uk}
\address[hap]{
Mathematical Institute, University of Oxford, Radcliffe Observatory Quarter,
Oxford OX2 6GG, United Kingdom}

\cortext[cor1]{corresponding author}



\begin{abstract}  
Bilattices (that is, sets with two lattice structures)
 provide an algebraic tool to model  simultaneously the validity of,  and  knowledge about,  
sentences in an appropriate language.
In  particular,   certain bilattices have been used to model situations 
in which 
information is prioritised and so can be viewed hierarchically.
 These default bilattices are not interlaced:
the  lattice operations of one lattice structure do not preserve the 
order of the other one.  
The well-known product representation theorem for interlaced
 bilattices does not extend to bilattices which fail to be interlaced 
and the lack of a product representation  has been a handicap 
to  understanding the  structure of default bilattices. 
In this paper we study, from an algebraic perspective, 
  a hierarchy  of varieties of  default bilattices,  
allowing for different levels of default.   
We develop  natural dualities for these varieties  and thereby 
obtain a  concrete representation for the  algebras in each variety. 
This leads on to a form of product representation that generalises the product representation 
as this  applies to 
distributive bilattices.
\end{abstract}

\end{frontmatter} 


\section{Introduction}\label{intro}

Our objective is to develop  a representation  theory 
for classes of algebras  which have arisen 
in the modelling of default logics.  
Specifically, we consider bilattices which have been used to
study logics with prioritised defaults~\cite{Rei80};   
the simplest and best known of these bilattices  was  introduced 
by Ginsberg \cite{Gins86} under the name $\SEVEN$.  
As we indicate below, 
such  `default bilattices' do not have the interlacing property  
and  so the equational classes they generate fall outside the scope 
of the Product Representation Theorem,
the cornerstone of the structure theory of interlaced bilattices.
A novel approach is  required in order to 
develop an analogous structure theory
beyond the interlaced setting. 
This we  provide by the application of natural duality theory.  
In \cite{CP1}, Cabrer and Priestley showed that, 
for the class $\DB$ of distributive bilattices, 
the product representation can be seen as a consequence of, 
and very closely allied to, 
 the natural duality for $\DB$ presented there. 
In the present paper we consider an infinite sequence 
of default bilattices, each having its predecessor
as a homomorphic image. We develop natural dualities  for  the equational classes  
generated by these bilattices and thereby arrive at 
a product representation for the members of these classes (Theorem~\ref{Thm:ProdRepDefault}).

To set the scene we recall the background very briefly. 
The motivation for Ginsberg's pioneering paper~\cite{Gins86} 
was his plan to use bilattices as a framework 
for inference with applications to artificial intelligence and 
logic programming, in particular for modelling inference  
in situations where information is  incomplete or contradictory. 
 The central idea  was to consider  sets  which carry 
two lattice orders: $\leq_t$, interpreted as measuring 
`degree of truth',  and $\leq_k$, measuring `degree of knowledge'. 
Certain elements of such structures were 
then treated as distinguished constants, 
representing degrees of truth or knowledge, 
$\mbt$ (`true'),  $\mbdf$ (`false  by default'),  and so on;
$\top$ and $\bot$ are used to denote, respectively, 
`contradiction' and `no information'.  
Fig.~\ref{fig:FOUR+SEVEN}(ii) shows the bilattice $\SEVEN$ 
Ginsberg proposed to model  this scenario.  
 As is customary in the bilattices literature 
the two constituent lattices are combined into a single diagram,
 with knowledge measured vertically and truth horizontally.

The bilattice $\SEVEN$ may be seen as providing a 
more refined  model  of  truth and falsity than 
the best-known bilattice of all, commonly known as $\FOUR$ 
and shown  in  Fig.~\ref{fig:FOUR+SEVEN}(i).  
In $\FOUR$, the elements $\mbt$ and $\mbf$ represent 
`true' and `false', $\top$ and $\bot$  `contradiction' and `no information'.

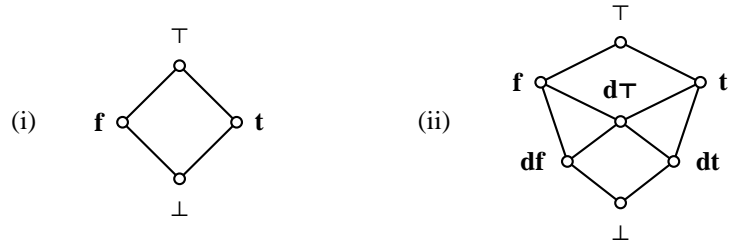
\begin{figure}[ht]
\begin{center}
\begin{tikzpicture}
\li{(-2,-0.5)--(-1.25,0.25)--(-0.5,-0.5)--(-1.25,-1.25)--(-2,-0.5)}
\node at (-1.25,0.65) {$\top$};
\node at (-2.3,-0.5) {$\mbf$};
\node at (-0.2,-0.5) {$\mbt$};
\node at (-1.25,-1.65) {$\bot$};

\epo{-2,-0.5}
\epo{-1.25,0.25}
\epo{-0.5,-0.5}
\epo{-1.25,-1.25}

\node at (-3.3,-0.5) {(i)};

\li{(4.55,0.56)--(3.5,0.03)--(4.55,-0.49)--(5.6,0.03)--(4.55,0.56)}
\li{(4.55,-1.57)--(3.85,-1.02)--(4.55,-0.49)--(5.25,-1.02)--(4.55,-1.57)}
\li{(3.5,0.03)--(3.85,-1.02)}
\li{(5.6,0.03)--(5.25,-1.02)}

\epo{4.55,0.56}
\epo{3.5,0.03}
\epo{5.6,0.03}

\epo{4.55,-0.49}
\epo{3.85,-1.02}
\epo{5.25,-1.02}
\epo{4.55,-1.57}

\node at (4.55,0.96) {$\top$};
\node at (4.55,-1.97) {$\bot$};
\node at (3.2,0.03) {$\mbf$};
\node at (5.9,0.03) {$\mbt$};
\node at (3.4,-1.02) {$\mbdf$};
\node at (5.7,-1.02) {$\mbdt$};
\node at (4.55,-0.06) {$\mbdT$};

\node at (2.1,-0.5) {(ii)};

\end{tikzpicture}
\caption{(i) the (distributive) bilattice $\FOUR$  and (ii) Ginsberg's bilattice $\SEVEN$ \label{fig:FOUR+SEVEN}}
\end{center}
\end{figure}

The bilattice $\SEVEN$ models one level of default.  
But there are situations in which a hierarchy of  
degrees of default may  be appropriate. 
Bilattices which model  prioritised defaults   
were discussed by Ginsberg~\cite[Section 7.3.3]{Gins88} 
and there is now a range of applications of  
such structures in artificial intelligence.  
We note for example the design by 
Encheva and Tumin~\cite{Encheva10} of a 
tutoring feedback system based on a ten-element default bilattice 
 to inform follow-up questions when the initial responses 
are incomplete or inconsistent.
The same ten-element bilattice is employed by Sakama~\cite{Sak}.  
Prioritised default bilattices have also been applied to visual
 surveillance by Shet, Harwood and Davis~\cite{Shet2006}. 
In Section~\ref{sec:ClassKn} we introduce an infinite sequence of
 bilattices $\K_n$, as a means of modelling prioritised defaults.
Here $\K_0$, $\K_1$ and $\K_2$
are the four-, seven- and ten-element bilattices mentioned above
(equipped with a negation and appropriate constants).
Fig.~\ref{fig:Kn-t-order} depicts the knowledge 
and truth orders of $\K_n$, for general~$n$.

There  is a critical difference between $\FOUR$ and $\SEVEN$.  
In $\FOUR$, each of the four lattice operations 
distributes over each of the other three; 
in $\SEVEN$ (and also in the refinements we consider) this fails. 
 A bilattice with lattice operations  
$\{\oplus,\otimes\}$   (with associated order $\leq_k$) 
 and $\{\vee,\wedge\}$ (with associated order $\leq_t$)
 is \defn{interlaced} if each pair of lattice operations is monotonic
 with respect to the other order.  
This holds  in $\FOUR$. 
But  in $\SEVEN$ it fails, as is witnessed by   the fact that
 $\mbdT \le_k \mbt$ and 
$\mbdT \wedge \bot = 
\mathbf{df} \nleqslant_k \bot = \mbt \wedge \bot$.
The significance of the interlacing condition  is that it is
sufficient,  and also necessary, for the Product Representation 
Theorem to be valid:
any interlaced bilattice has as 
its underlying set a product $L \times L$ of a lattice~$L$ with itself; 
the lattice operations on the factors determine the bilattice
operations; negation sends a pair $(a,b)$ to $(b,a)$.  
For an account of the theorem  and its  complicated history,
see the recent note by Davey~\cite{BD13}.
This note gives a comprehensive list of  references both to the theorem itself
and to the way in which it is used to study interlaced bilattices. 

We conclude this introduction by summarising the content 
and structure of the paper and highlighting our principal results.  
We  focus  on mathematical aspects of default bilattices, 
 rather than logical aspects. 
 We shall consider $\K_n$ as an algebra of a specified type 
and investigate the variety  $\CV_n=\HSP(\K_n)$
 generated by $\K_n$, for an arbitrary value of~$n$.  
In Section~\ref{sec:ClassKn} we  derive the  properties 
of the algebras $\K_n$ on which our representation theory will rely.

Our primary tool, as in \cite{CP1}, will be the theory of
natural dualities, for which the text by Clark and Davey \cite{CD98} 
serves as the  background reference.
Here we need the multisorted version of the theory, 
as it applies to a restricted class of finitely generated 
lattice-based varieties. 
 Section~\ref{sec:multisorted} outlines, as far as possible in 
black-box style, rudiments of  this theory.  
The framework was  first developed more than~25 years ago, 
but examples of its exploitation are quite scarce.  
Theorem~\ref{Thm:DualityVarietyKn} 
describes our duality for~$\CV_n$,
an instructive new example of  the multisorted machinery at work. 
It also provides  us with a  springboard to our later results. 
 
In Theorem~\ref{Thm:CharDualVn} we 
describe the objects in the category  
dual  to~$\CV_n$: 
 these  multisorted topological structures 
are such that  each sort naturally   carries 
the  structure of a Priestley space, and  there is a sequence of maps
which links each sort to the next. 
In  Section~\ref{sec:ProdRep}, we derive 
 our  product representation theorem 
(Theorem~\ref{Thm:ProdRepDefault}).
The proof makes explicit use of 
the multisorted structure dual to a given algebra in $\CV_n$
to show how  the algebra can be obtained  from 
a product built from 
a finite sequence
 of distributive lattices with linking homomorphisms.    
Section~\ref{sec:ProdRep} can if desired be studied independently of
Section~\ref{sec:RevEng}. In the latter we connect directly
with Priestley duality, relating the natural dual space of an algebra
 in $\CV_n$ to  the Priestley dual of its (necessarily distributive)
knowledge lattice reduct.    
The paper concludes with a short section devoted to the
 quasivarieties $\ISP(\K_n)$. 
Here we can  employ duality theory in its single-sorted form.

\section{A hierarchy of varieties of prioritised default bilattices}
\label{sec:ClassKn}

The term `bilattice' is not used in a consistent way 
throughout the extensive bilattice literature.  
However nowadays it usually refers to  a structure which, 
besides  its two lattice structures  and, possibly, constants, 
carries also a negation operation, and we follow this usage. 
 Henceforth a \defn{bilattice} $\B$ will be  an algebraic structure 
$\B=(B; \otimes, \oplus,\wedge,\vee, \neg )$ 
such that the reducts $( B; \otimes,\oplus)$ and 
$(B; \wedge, \vee )$ are lattices and $\neg$ is a unary operation
 which preserves the $\{\otimes,\oplus\}$-order, 
reverses the $\{\wedge,\vee\}$-order, and is involutive.
We denote by $\le_k$ the order associated with 
$( B; \otimes, \oplus)$ and $\le_t$ the order associated with
 $(B; \wedge, \vee )$. 
We shall consider only bilattices in which 
both of the lattice orders are bounded. 
The bounds of the 
knowledge lattice are denoted by $\top$ and $\bot$, and those  
of the truth lattice by $\mbt$ and $\mbf$.

In an $n$-level prioritised default bilattice, 
the designated  default truth values form two finite sequences 
$ \mbf_0, \dots,  \mbf_n$ and  $\mbt_0, \dots, \mbt_n$, 
where $\mbf_{k+1}$ and $\mbt_{k+1}$  will
be lower in the knowledge order than their respective 
predecessor default truth values  $\mbf_k$ and $\mbt_k$.
The connotation is thus that knowledge represented by 
the truth values at level 
$k+1$ has lower priority than that from those at level~$k$. 
In addition, one thinks of  $\mbt_{k+1}$ as being `less true' than 
its predecessor $\mbt_{k}$, while 
$\mbf_{k+1}$ is `less false' than $\mbf_{k}$.  
That is, $\mbt_{k+1} \le_t \mbt_k$ and
 $\mbf_{k+1} \ge_t \mbf_{k}$.
Thus we view the  truth values hierarchically.

We now  describe  the $n$-level 
prioritised default bilattice $\K_n$
for $n \ge 0$.  
The underlying set of this algebra is 
\[
K_n = \{ 
\mbf_0,\ldots, \mbf_n, \mbt_0,\ldots,\mbt_n, {\top}_{\!\!0}, \ldots,
{\top}_{\!\!n+1}\}.
\]
We 
define lattice orders $\le_k$ and $\le_t$ on $K_n$ as follows.  
For $K_0$ and $v \in \{\mbf_0,\mbt_0\}$ we have $\top_{\!\!1} <_k v <_k \top_{\!\!0}$.
If $n \ge 1$ and 
$0 \le i < j \le n$, 
\begin{align*} 
&\top_{\!\!n+1} <_k  v_j <_k \top_{\!\! j} <_k 
v_i <_k \top_{\!\!i} \quad \text{for  $ v_j \in \{\mbf_j,\mbt_j\}$  and 
$v_i \in 
\{\mbf_i,\mbt_i\} $}, 
\intertext{
and if $0 \le i \le j \le n$,}
&\mbf_i \le_t \mbf_j <_t \top_{\!\!j} <_t \mbt_j \le_t \mbt_i;\,\mbox{ and  } \mbf_i <_t \top_{\!\!n+1} <_t \mbt_i. 
\end{align*} 

These lattice orders are depicted in Hasse diagram style 
in Fig.~\ref{fig:Kn-t-order}.
We shall, where appropriate, write  $\top$ for $\top_{\!\!0}$ 
and~$\bot $ for $\top_{\!\!n+1}$.

\begin{figure}[ht]
\begin{center}
\begin{tikzpicture}[scale=1]
\li{(1.25,1.25)--(2,2)--(2.75,1.25)}
\dotli{(1.25,1.25)--(2,0.5)--(2.75,1.25)}
\li{(1.25,-0.25)--(2,0.5)--(2.75,-0.25)}
\dotli{(1.25,-0.25)--(2,-1)--(2.75,-0.25)}
\li{(1.25,-1.75)--(2,-1)--(2.75,-1.75)}
\dotli{(1.25,-1.75)--(2,-2.5)--(2.75,-1.75)}
\li{(1.25,-3.25)--(2,-2.5)--(2.75,-3.25)--(2,-4)--(1.25,-3.25)}

\epo{2,2}
\epo{1.25,1.25}
\epo{2.75,1.25}
\epo{2,0.5}

\epo{1.25,-0.25}
\epo{2.75,-0.25}
\epo{2,-1}

\epo{1.25,-1.75}
\epo{2.75,-1.75}
\epo{2,-2.5}

\epo{1.25,-3.25}
\epo{2.75,-3.25}
\epo{2,-4}

\node at (2.7,2.1) {$\top_{\!\!0}=\top$};
\node at (0.8,1.2) {$\mbf_0$};
\node at (3.2,1.2) {$\mbt_0$};
\node at (2,0.9) {$\top_{\!\!i}$};
\node at (0.8,-0.3) {$\mbf_i$};
\node at (3.2,-0.3) {$\mbt_i$};
\node at (2.0,-0.65) {$\top_{\!\!j}$};
\node at (0.8,-1.75) {$\mbf_j$};
\node at (3.2,-1.75) {$\mbt_j$};
\node at (2.0,-2.1) {$\top_{\!\!n}$};
\node at (0.8,-3.25) {$\mbf_n$};
\node at (3.2,-3.25) {$\mbt_n$};

\node at (3,-4.1) {$\top_{\!\!n+1}=\bot$};

\dotli{(7.5,1)--(9.5,0)}
\dotli{(7.5,-3)--(9.5,-2)}
\li{(9.2,-2.15)--(9.2,0.15)}
\li{(10.5,-1.5)--(10.5,-0.5)}
\li{(10.5,-0.5)--(11.5,-1)--(10.5,-1.5)}
\li{(7.8,0.85)--(7.8,-2.85)}
\li{(6.5,1.5)--(6.5,-3.5)}
\dotli{(6.5,1.5)--(7.5,1)}
\dotli{(6.5,-3.5)--(7.5,-3)}
\dotli{(9.5,0)--(10.5,-0.5)}
\dotli{(9.5,-2)--(10.5,-1.5)}

\epo{9.2,0.15}
\epo{9.2,-1}
\epo{10.5,-0.5}
\epo{10.5,-1.5}
\epo{10.5,-1}
\epo{11.5,-1}
\epo{9.2,-2.15}

\epo{7.8,0.85}
\epo{7.8,-2.85}
\epo{7.8,-1}
\epo{6.5,1.5}
\epo{6.5,-3.5}
\epo{6.5,-1}

\node at (6.75,1.7) {$\mbt_0$};
\node at (5.7,-1) {$\top\!=\!\top_{\!\!0}$};
\node at (6.75,-3.7) {$\mbf_0$};

\node at (8.2,0.95) {$\mbt_i$};
\node at (7.4,-1) {$\top_{\!\!i}$};
\node at (8.2,-3.05) {$\mbf_i$};

\node at (9.3,0.45) {$\mbt_{j}$};
\node at (8.8,-1) {$\top_{\!\!j}$};
\node at (9.25,-2.55) {$\mbf_{j}$};

\node at (10.75,-1.8) {$\mbf_{n}$};
\node at (10.8,-0.3) {$\mbt_{n}$};
\node at (10.1,-1) {$\top_{\!\!n}$};

\node at (12.32,-1) {$\top_{\!\!n+1}\!=\!\bot$};

\end{tikzpicture}
\caption{
$K_n$
in its knowledge order (left) and truth order (right);  here  $0 < i < j < n$ \label{fig:Kn-t-order}}
\end{center}
\end{figure}
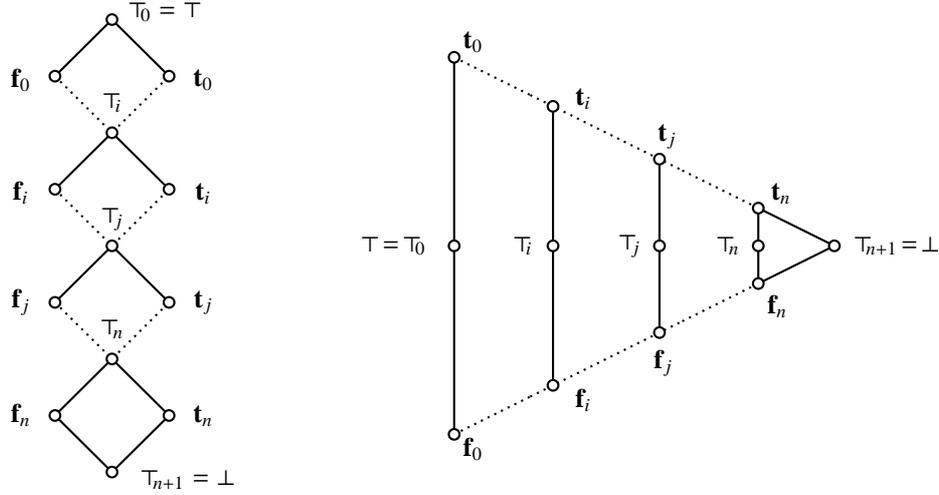

Negation is defined as follows:
\[
\neg \mbt_i= \mbf_i \text{ and } \neg\mbf_i = \mbt_i  \ \text{(for  
 $0 \leq i \le n$)}; \quad  
\neg a = a \text{ if } a \in \{ \top_{\!\!0}, \ldots, \top_{\!\!n+1}\}.
\]
The elements 
$\mbt_i$ and $\mbf_i$ are inter-definable via $\neg$.  
Elements of these types
contain only information about truth or falsity; they 
do not contain any information about contradiction or 
lack of information.
We observe that every element of $K_n$ is
recursively term-definable from 
$\top_{\!\!0}$ and $\top_{\!\!n+1}$.
For $m \in \{0,\ldots,n\}$ we have 
\[
\mbt_m=\top_{\!\!m} \vee \top_{\!\!n+1},\ \mbf_m=\top_{\!\!m} \wedge \top_{\!\!n+1} \quad \mbox{and} \quad 
\top_{\!\!m+1}=(\top_{\!\!m} \vee \top_{\!\!n+1})\otimes(\top_{\!\!m} \wedge \top_{\!\!n+1}).
\]
Note in particular  that  $\mbf$ and $\mbt$ are 
 $\top\wedge\bot$ and $\top\vee\bot$ respectively.

We now define  the algebra $\K_n$ to be  
$\K_n = ( K_n; \otimes, \oplus, \wedge, \vee, \neg, \bot,\top)$.
It has  a term-definable bounded bilattice structure.   
The cases $n=0$ and $n=~1$ 
 deserve special mention.  The bilattice reducts of $\K_0$
and $\K_1$ are, up to the labelling of the elements, simply 
$\mathcal{FOUR}$ and $\mathcal{SEVEN}$.  Moreover, $\K_1$ is term-equivalent to the algebra 
${\four=(\{\top,\bot,\mbt,\mbf\};\wedge, \vee, \neg, \mbt, \mbf,\bot,\top)}$ 
introduced in \cite[Section~2]{CP1}; 
here $\neg$ switches $\mbt$ and $\mbf$, 
and fixes $\top$ and $\bot$. 
%
%
The algebras $\K_n$ are all of the same algebraic type.  
All have the special property that their knowledge reducts 
are bounded distributive lattices and 
the same is true of the algebras in  $\CV_n$, 
 for any $n$. 
But for $n> 0$ the truth lattice reduct of~$\K_n$ is not distributive.  
In what follows we shall take a fixed $n \geq 0$;
we include $n=0$ to emphasise that $\K_0$ 
fits into our general scheme. 
However  our results below give nothing new in this case.

In our development of the properties of 
the algebras $\K_n$ and the varieties $\CV_n$ that they generate, 
we shall occasionally need to draw on  basic  
facts from universal algebra, for example concerning 
congruences and subdirectly irreducible algebras; 
\cite{BS81} provides a good background reference for such material.
Our first observation is a triviality:  since each element of $\K_n$ 
is a term, $\K_n$ has no proper subalgebras.  
We shall next investigate  the subalgebras of  $\K_n^2$.  
The characterisation we obtain in Theorem~\ref{thm:Snmonly} 
will be of  crucial importance for setting up our dualities.  
As a byproduct, we are able to  identify 
the subdirectly irreducible algebras in  the variety $\CV_n$ 
and thence obtain  a complete description of its 
lattice of subvarieties.

We let $\Delta_n$ denote the diagonal subalgebra  
$\{\, (a,a) \mid a\in K_n\,\}$.  
We now define subsets $S_{n,n}, \dots, S_{n,0}$ of $K_n^2$ 
as follows:
\[
 S_{n,m} = \Delta_n \cup 
\{\, (a,b)  \mid a, b \le_k\! {\top}_{\!\!m+1} \mbox{ or }a \le_k b \le_k \!{\top}_{\!\!m}\,\}  \quad \text{(for  $0 \leq m  \le n$)}.
 \]
For $0 \le i < j \leq n$ we have $S_{n,j} \subsetneqq S_{n,i}$.

Later we shall want to view 
 $S_{n,m}$ as a 
binary relation on $K_n$.  
It is easily seen that, as such, it is always a quasi-order, 
and a partial order if and only if $m=n$.
The partial orders $S_{m,m}$, for $m \leq n$,
appear in  our duality theory for $\CV_n$ in Section~\ref{sec:varKn} and 
the relations $S_{n,m}$, for $m\leq n$,  are employed in the duality 
for $\ISP(\K_n)$ presented in Section~\ref{sec:dualitiesKn}.  
By way of illustration, Fig.~\ref{fig:Snm} shows $S_{0,0}$, $S_{1,1}$ and $S_{2,2}$, and also 
$S_{2,0}$ and $S_{2,1}$. 

\begin{figure}[ht]
\begin{center}
\begin{tikzpicture}[scale=0.56]
\path (-18.5,-6) node(ai) [rectangle,draw, minimum size=1.5pt, inner sep=3pt] {$\scriptstyle{\top_{\!\!1}}$}
			(-20,-4.5) node(aj) [rectangle,draw, minimum size=1.5pt, inner sep=3pt] {$\scriptstyle{\mbf_0}$}
			(-17,-4.5) node(ak) [rectangle,draw, minimum size=1.5pt, inner sep=3pt] {$\scriptstyle{\mbt_0}$}
			(-18.5,-3) node(al) [rectangle,draw, minimum size=1.5pt, inner sep=3pt] {$\scriptstyle{\top_{\!\!0}}$}
			(-21.3,-6) node (am) {$S_{0,0}$}
			(-9.5,-6) node(a) [rectangle,draw, minimum size=1.5pt, inner sep=3pt] {$\scriptstyle{\top_{\!\!2}}$}
			(-11,-4.5) node(b) [rectangle,draw, minimum size=1.5pt, inner sep=3pt] {$\scriptstyle{\mbf_1}$}
			(-8,-4.5) node(c) [rectangle,draw, minimum size=1.5pt, inner sep=3pt] {$\scriptstyle{\mbt_1}$}
			(-9.5,-3) node(d) [rectangle,draw, minimum size=1.5pt, inner sep=3pt] {$\scriptstyle{\top_{\!\!1}}$}
			(-11,-2) node(e) [rectangle,draw, minimum size=1.5pt, inner sep=3pt] {$\scriptstyle{\mbf_0}$}
			(-8,-2) node(f) [rectangle,draw, minimum size=1.5pt, inner sep=3pt] {$\scriptstyle{\mbt_0}$}
			(-9.5,-1) node(g) [rectangle,draw, minimum size=1.5pt, inner sep=3pt] {$\scriptstyle{\top_{\!\!0}}$}
			(-12.3,-6) node(h) {$S_{1,1}$}
			(-0.5,-6) node(i) [rectangle,draw, minimum size=1.5pt, inner sep=3pt] {$\scriptstyle{\top_{\!\!3}}$}
			(-2,-4.5) node(j) [rectangle,draw, minimum size=1.5pt, inner sep=3pt] {$\scriptstyle{\mbf_2}$}
			(1,-4.5) node(k) [rectangle,draw, minimum size=1.5pt, inner sep=3pt] {$\scriptstyle{\mbt_2}$}
			(-0.5,-3) node(l) [rectangle,draw, minimum size=1.5pt, inner sep=3pt] {$\scriptstyle{\top_{\!\!2}}$}
			(-0.5,-1) node(m) [rectangle,draw, minimum size=1.5pt, inner sep=3pt] {$\scriptstyle{\top_{\!\!1}}$}
			(1,-2) node(n) [rectangle,draw, minimum size=1.5pt, inner sep=3pt] {$\scriptstyle{\mbt_1}$}
			(-2,-2) node(o) [rectangle,draw, minimum size=1.5pt, inner sep=3pt] {$\scriptstyle{\mbf_1}$}
			(-2,0) node(p) [rectangle,draw, minimum size=1.5pt, inner sep=3pt] {$\scriptstyle{\mbf_0}$}
			(1,0) node(q) [rectangle,draw, minimum size=1.5pt, inner sep=3pt] {$\scriptstyle{\mbt_0}$}
			(-0.5,1) node(r) [rectangle,draw, minimum size=1.5pt, inner sep=3pt] {$\scriptstyle{\top_{\!\!0}}$}
			(-3.3,-6) node(u) {$S_{2,2}$}
			(-14,-14) node(v) [rectangle,draw, minimum size=1.5pt, inner sep=3pt] {$\scriptstyle{\top_{\!\!2},\mbf_2,\mbt_2,\!\top_{\!\!3}}$}
			(-15.5,-12.5) node(w) [rectangle,draw, minimum size=1.5pt, inner sep=3pt] {$\scriptstyle{\mbf_1}$}
			(-12.5,-12.5) node(x) [rectangle,draw, minimum size=1.5pt, inner sep=3pt] {$\scriptstyle{\mbt_1}$}
			(-14,-11) node(y) [rectangle,draw, minimum size=1.5pt, inner sep=3pt] {$\scriptstyle{\top_{\!\!1}}$}
			(-15.5,-10) node(z) [rectangle,draw, minimum size=1.5pt, inner sep=3pt] {$\scriptstyle{\mbf_0}$}
			(-12.5,-10) node(aa) [rectangle,draw, minimum size=1.5pt, inner sep=3pt] {$\scriptstyle{\mbt_0}$}
			(-14,-9) node(ab) [rectangle,draw, minimum size=1.5pt, inner sep=3pt] {$\scriptstyle{\top_{\!\!0}}$}
			(-17,-14) node(ac) {$S_{2,1}$}
			(-5,-14) node(ad) [rectangle,draw] {$\scriptstyle{\top_{\!\!1},\mbf_1,\mbt_1,\!\top_{\!\!2},\mbf_2,\mbt_2,\!\top_{\!\!3}}$}
			(-6.5,-12.5) node(ae) [rectangle,draw, minimum size=1.5pt, inner sep=3pt] {$\scriptstyle{\mbf_0}$}
			(-3.5,-12.5) node(af) [rectangle,draw, minimum size=1.5pt, inner sep=3pt] {$\scriptstyle{\mbt_0}$}
			(-5,-11) node(ag) [rectangle,draw, minimum size=1.5pt, inner sep=3pt] {$\scriptstyle{\top_{\!\!0}}$}
			(-8.2,-14) node(ah) {$S_{2,0}$};
\draw (ai)--(aj)--(al)--(ak)--(ai);
\draw (a)--(b)--(d)--(c)--(a);
\draw (i)--(j)--(l)--(k)--(i);
\draw (v)--(w)--(y)--(x)--(v);
\draw (ad)--(ae)--(ag)--(af)--(ad);
\end{tikzpicture}
\caption{The relations $S_{0,0}$, $S_{1,1}$, $S_{2,2}$, $S_{2,1}$ and $S_{2,0}$ drawn as quasi-orders\label{fig:Snm}}
\end{center}
\end{figure}
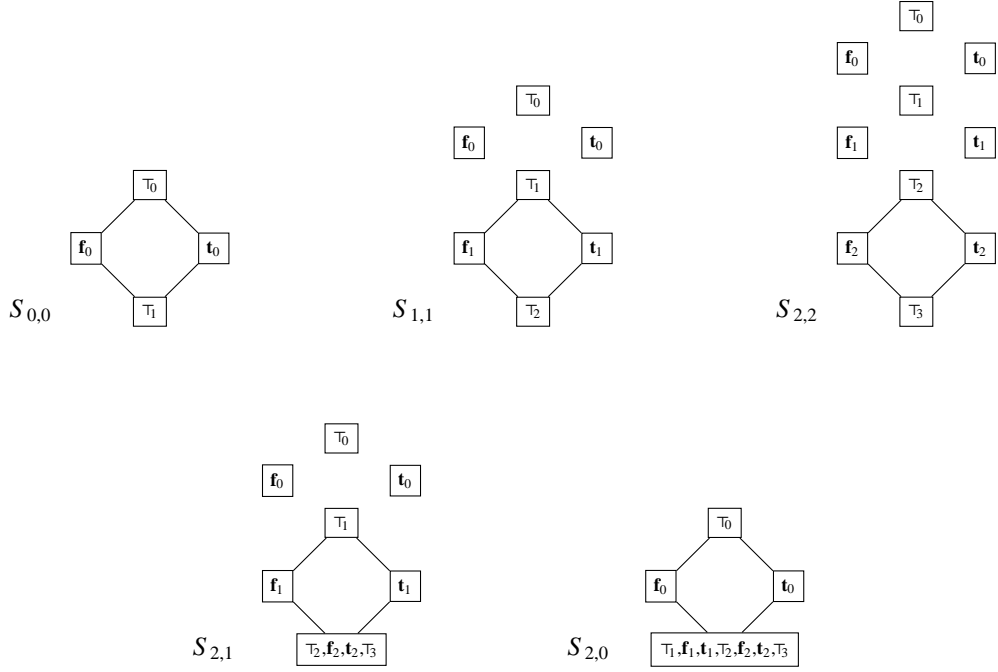

For $m \leq n$ we let $h_{n,m} \colon K_n \to K_m$ be given, 
for $a \in K_n$,  by  
\[
h_{n,m}(a)=
\begin{cases}
\top_{\!\!m+1} &\text{if } a \le_k \top_{\!\!m+1},\\
a&\text{otherwise}.
\end{cases}
\]
Here $h_{n,n}$ is just the identity map on $\K_n$.  

\begin{prop}\label{prop:Snn-subalg} 
For $m$ such that\, $0 \le m\le n$, 
let $S_{n,m} $ and $h_{n,m}$ be defined as above.  Then the following statements hold. 
\begin{newlist}
\item[{\upshape (i)}] 
$S_{n,m} = \{\,(a,b)\in K_n^2 \mid  (h_{n,m}(a),h_{n,m}(b))\in S_{m,m} \,\}$.
\item[{\upshape (ii)}]  $h_{n,m}$ is a surjective homomorphism.
\item[\upshape {(iii)}]  
$S_{n,m}$, with the inherited operations, forms a subalgebra\, $\S_{n,m}$ of\, $\K_n^2$.
\item[{\upshape (iv)}] 
$\mathbb{S}(\K_{i}\times\K_{j})=
\{(h_{n,i}\times h_{n,j})(\S)\colon \S\in\mathbb{S}(\K_n^2)\}$, 
for $i,j \in \{0,\ldots,n\}$.

\end{newlist}
\end{prop}
\begin{proof}  
Part (i) is straightforward. We  now prove (ii). 
If $b \le_k{ \top}_{\!\!m+1}$ we have  
$\neg b \le_k {\top}_{\!\!m+1}$.  
Hence $\neg h_{n,m}(b)= 
\neg (\top_{\!\!m+1})=\top_{\!\!m+1} = h_{n,m}(\neg b)$.
Obviously $h_{n,m}$ preserves the constants $\top$ and $\bot$.
It now suffices to show that 
\[
\ker h_{n,m}=
\Delta_n \cup\{\,(a,b)\mid a,b\le_k {\top}_{\!\!m+1}\,\}
\]
is a lattice congruence for both lattice orders on $K_n$.
This can be done by a routine check that the equivalence 
classes are convex sublattices satisfying the quadrilateral property 
(see for example \cite[Chapter~6]{ILO2}).

We now prove (iii).
Consider  first the case $m=n$. 
We recall that $S_{n,n}=
\Delta_n \cup \{\,(a,b) \mid a \le_k b\le_k \top_{\!\!n}\,\}$.
Fig.~\ref{fig:Kn-t-order} provides a useful guide to performing
 bilattice operations on elements of $S_{n,n}$. 
Let $(a,b)$, $(c,d)\in S_{n,n}$. 
If $a,b,c,d\le_k {\top}_{\!\!n}$, then 
$a\star c\le_k b\star d\le_k {\top}_{\!\!n}$ for 
$\star\in\{\otimes,\oplus,\wedge,\vee\}$, 
that is, $(a\star c, b\star d)\in S_{n,n}$. 
If $a>_{k}{\top}_{\!\!n}$ then $b=a$. If $a=\mbf_i$ or $a=\mbt_i$ for $i <n$, then it is straightforward to check that 
$(a\star c,a \star d) \in \{(a,a),(c,d)\} \subseteq S_{n,n}$ 
for $\star\in\{\otimes,\oplus,\wedge,\vee\}$.  
If $a>_{k}{\top}_{\!\!n}$ and $a={\top}_{\!\!i}$ for $i < n$, then 
$(a\star c,a \star d) \in \{(a,a),(c,d),(\mbt_i,\mbt_i),(\mbf_i,\mbf_i)\} \subseteq S_{n,n}$ depending on whether $\star$ is 
$\oplus,\otimes, \vee$ or $\wedge$. 
Thus  $S_{n,n}$ is closed under 
 $\otimes$, $\oplus$, $\wedge$, and $\vee$. 
Also, if $(a,b)\in  S_{n,n}$ and $a\neq b$ then 
$ a\le_k b\le_k{\top}_{\!\!n}$. 
Since $\neg$ preserves $\le_k$, it follows that   
$ \neg a\le_k \neg b\le_k\neg {\top}_{\!\!n}={\top}_{\!\!n}$, 
that is, $(\neg a,\neg b)\in S_{n,n}$. 
Finally,  the pairs $(\bot,\bot)$ and $(\top,\top)$
are in $S_{n,n}$. 
So $\S_{n,n}\in\mathbb{S}(\K_n^2)$ for each $n$.  

Now assume $m < n$. 
The product map 
$(h_{n,m} \times h_{n,m}) \colon \K_n^2\to\K_m^2$ defined by
 $(a,b)\mapsto (h_{n,m}(a),h_{n,m}(b))$ is a homomorphism. 
By~(i), $S_{n,m}$ is the inverse image of $S_{m,m}$ under 
$h_{n,m} \times h_{n,m}$.
Using the fact that 
$\S_{m,m} \in \mathbb{S}(\K_m^2)$, we 
deduce that $\S_{n,m} \in \mathbb{S}(\K_n^2)$. 

To prove (iv)  observe that the product map
$(h_{n,i}\times h_{n,j})\colon \K_{n}^2\to \K_{i}\times\K_{j}$ 
is a surjective homomorphism.
\end{proof}

We are now ready to identify all the subalgebras of $\K_n^2$.  We observe 
that, for $n > 0$,  neither $\leq_k$ nor $\leq_t$ is the universe of such an algebra.
This can be attributed to  the failure of the interlacing condition.

\begin{thm}\label{thm:Snmonly}
Let\, $\S$ be a subalgebra of\, $\K_n^2$.  Then either\, $\S = \K_n^2$ or there exists  $m \in \{0,\ldots,n\}$ 
such that \,$\S$ is one of the following:
$ \S_{n,m}$,  $\conv{\S}_{n,m}$,
$\S_{n,m} \cap \conv{\S}_{n,m}$.
{\upshape(}For $r$ a binary relation, $\conv{r}$ denotes its 
converse.{\upshape)}\end{thm}

\begin{proof}
The fact that $\K_n$ is generated by $\{\bot,\top\}$ implies that $\Delta_n$ is generated by $\{(\bot,\bot),(\top,\top)\}$. Hence $\Delta_n$ is contained in every subalgebra of $\K_n^2$.
Let 
\[ 
A = \{\top_{\!\!n+1}\}\cup
\{\, a \mid \exists\, b  \text{ such that }a>_k b\mbox{ and } (a,b) \in S \,\} \text{ \ \ and \ \ } 
B=\{\top_{\!\!n+1}\}\cup 
\{\, b \mid \exists\,a \text{ such that } a<_k b\mbox{ and }(a,b) \in 
S\,\}.\]
Claims 1--4 below concern $A$.   These claims, together with 
corresponding results for $B$ obtained by swapping the coordinates, 
will be combined to prove the theorem.

\begin{claimslist}
\item[{\bf Claim 1}:] 
$\bigoplus A\in A$. 

If $A=\{\top_{\!\!n+1}\}$ 
the result is trivial. If $a,b \in A$ are such that $a\not\le_k b $ 
and $b \not\le_k a$ then there exists $i\in\{1,\ldots n\}$ such that 
$\{a,b \}=\{\mbf_i,\mbt_i\}$. Let $b,b'$  be such that 
$(a,b)$, $(b ,b')\in
S$, $b<_k a$ and $b'<_k b $. 
Then $b\oplus b'\le_k \top_{\!\!i+1}<_k \top_{\!\!i}=a\oplus b $. 
We also have $(a\oplus b ,b\oplus b') = (a,b)\oplus(b ,b')\in
S$, 
Hence $a\oplus b \in A$. Thus $A$ is a finite set closed under
$\oplus$, and consequently  $\bigoplus A\in A$.  

\item[{\bf Claim 2}:]  
$\bigoplus A \in \{\top_{\!\!0},\ldots, \top_{\!\!n+1}\}$.

Since $S$ is closed under $\neg$ and $\neg$ preserves $\le_k$, 
we have $\mbf_i\in A$ if and only if $\mbt_i\in A$. 
This, combined with Claim 1, 
implies that if $\mbf_i\in A$ or $\mbt_i\in A$ 
then ${\top}_{\!\!i}\in A$.   

\item[{\bf Claim 3}:] 
If $\bigoplus A={\top}_{\!\!i}$ 
for some $i\in\{0,\ldots,n\}$, then 
$(\top_{\!\!i}, \mbf_i)$, $(\top_{\!\!i}, \mbt_i)$, $({\top}_{\!\!i},
{\top}_{\!\!i+1})$, $(\mbf_i,\top_{\!\!i+1})$ and  $(\mbt_i,\top_{\!\!i+1})$ 
all belong to $S$.

Let $b<_k {\top}_{\!\!i}$ be such that $(\top_{\!\!i},b)\in S$. 
If $b\in\{\mbf_i,\mbt_i\}$, since 
$S$ is closed under negation,
$({\top}_{\!\!i},\mbf_i),({\top}_{\!\!i},\mbt_i)\in S$. 
If $b\notin \{\mbf_i,\mbt_i\}$ then $b\le_k {\top}_{\!\!i+1}$, and 
$({\top}_{\!\!i},\mbf_i)=({\top}_{\!\!i},b)\oplus(\mbf_i,\mbf_i)\in S$
 and  
$({\top}_{\!\!i},\mbt_i)=({\top}_{\!\!i},b)\oplus(\mbt_i,\mbt_i)\in S$.
We also have 
$({\top}_{\!\!i},\top_{\!\!i+1})=
({\top}_{\!\!i},\mbf_i)\otimes({\top}_{\!\!i},\mbt_i)\in S$. 
And finally 
$ (\mbf_i,{\top}_{\!\!i+1})=({\top}_{\!\!i},{\top}_{\!\!i+1})\otimes(\mbf_i,\mbf_i),(\mbt_i,{\top}_{\!\!i+1})
=({\top}_{\!\!i},{\top}_{\!\!i+1})\otimes
(\mbt_i,\mbt_i) \in S$.

\item[{\bf Claim 4}:]
If $\bigoplus A={\top}_{\!\!i}$ for some 
$i\in\{0,\ldots,n\}$, then $(a,b)\in S$ for each 
$a\le_k {\top}_{\!\!i}$ and $b\le_k \top_{\!\!i+1}$. 

By Claim 3,  
$(\mbf_i,{\top}_{\!\!i+1}),(\mbt_i,{\top}_{\!\!i+1})\in S$ and 
\[
(\top_{\!\!n+1},{\top}_{\!\!i+1})=
((\top_{\!\!n+1},\top_{\!\!n+1})\vee(\mbf_i,{\top}_{\!\!i+1}))\oplus
 ((\top_{\!\!n+1},\top_{\!\!n+1})\wedge(\mbt_i,{\top}_{\!\!i+1}))\in S.
\]
Again by Claim 3, $({\top}_{\!\!i},\mbf_i),({\top}_{\!\!i},\mbt_i)\in S$, 
and then 
$({\top}_{\!\!i},\top_{\!\!n+1})=
((\top_{\!\!n+1},\top_{\!\!n+1})\vee({\top}_{\!\!i},\mbf_i))\oplus 
((\top_{\!\!n+1},\top_{\!\!n+1})\wedge({\top}_{\!\!i},\mbt_i))\in S$.
Finally, if  $a\le_k {\top}_{\!\!i}$ and $b\le_k {\top}_{\!\!i+1}$, then 
$(a,b)=((a,a)\oplus(\top_{\!\!n+1},{\top}_{\!\!i+1}))\otimes
((b,b)\oplus({\top}_{\!\!i},\top_{\!\!n+1}))\in S$.
\end{claimslist}

We are now ready to prove the main result.
 Let $i,j\in\{0,\ldots,n+1\}$ be such that 
$\bigoplus A={\top}_{\!\!i}$ and $\bigoplus B={\top}_{\!\!j}$.
We now have four  cases, taking  account of how~$i$ and~$j$ 
are related.

\begin{caseslist}
\item[{\bf Case 1}:]
Assume  $i=j =0$.  Then $\S=\K_n^2$.   

By Claims 3 and 4,  $S_{n,0}\,\cup\,\conv{S}_{n,0}\subseteq S$. 
Moreover 
$(\mbt_0,\mbf_0)
=(\mbt_0,\top_{\!\!1})\oplus(\top_{\!\!1},\mbf_0)\in S$ 
and similarly $(\mbf_0,\mbt_0)\in S$. 

\item[{\bf  Case 2}:]
Assume  $i=j>  0$. 
Then $\S=
\S_{n,i-1}\cap\conv{\S}_{n,i-1}$.

By definition of $A$ and $B$, we have 
$S\subseteq ({\downarrow}_k \top_{\!\!i})^2\cup\Delta_n=
S_{n,i-1} \cap \conv{S}_{n,i-1}$, where
${\downarrow}_k \top_{\!\!i}=\{c\in K_n\mid c\leq_k \top_{\!\!i}\}$.
By Claims 3 and 4, $S_{n,i}\cup\conv{S}_{n,i}\subseteq S$.
Moreover  $(\mbt_i,\mbf_i)=(\mbt_i,\!\top_{\!\!i+1})\oplus(\top_{\!\!i+1},\mbf_i)\in S$ 
and similarly $(\mbf_i,\mbt_i)\in S$. 
Thus 
$S\subseteq ({\downarrow}_k \top_{\!\!i})^2\cup\Delta_n\subseteq S$. 

\item[{\bf Case 3}:] 
Assume  $0\leq j<i\leq n+1$. Then  $j=i-1$ and  
$\S=\S_{n,j}$.

By  Claim 4,  $j=i-1$.
By definition of $A$ and $B$ and Claims 3 and 4, 
$S_{n,j}\subseteq S \subseteq({\downarrow}_k \top_{\!\!j})^2\cup\Delta_n$. 
But 
$\big(({\downarrow}_k \top_{\!\!j})^2\cup\Delta_n \big)
\setminus S_{n,j}=\{(\mbf_j,\mbt_j),(\mbt_j,\mbf_j)\}$. 
If
$(\mbf_j,\mbt_j)\in \S$, then
 $(\top_{\!\!j},\mbt_j)=(\mbf_j,\mbt_j)\oplus(\mbt_j,\mbt_j)\in S$, 
that is, $\top_{\!\!j}\in A$, 
which contradicts the assumption that $j<i$. 
A  contradiction is likewise obtained if we assume 
$(\mbt_j,\mbf_j)\in S$.  Thus $S=S_{n,j}$.

\item[{\bf Case 4}:] 
Assume  $0\leq i<j\leq n+1$.  
Then  $i=j-1$ and  
$\S=\conv{\S}_{n,j}$.

The proof is analogous to that  for  Case 3. \qedhere
\end{caseslist}
\end{proof}

\begin{coro}
\label{cor:congKn} 
 \begin{newlist}
\item[{\upshape (i)}] 
The congruence lattice\, $\mathrm{Con}(\K_n)$ of\, $\K_n$ 
is a chain with  $(n+2)$ elements.
\item[{\upshape (ii)}]   
The algebra\, $\K_n$ is subdirectly irreducible.  
\end{newlist} 
\end{coro}
\begin{proof}
For $m \in \{0,\ldots,n\}$, the relation $S_{n,m}$
is not a congruence on $\K_n$, as it is not symmetric and 
so not an equivalence relation.
On the other hand each relation
$S_{n,m} \cap \conv{S}_{n,m}$ is a congruence, and so  is $\K_n^2$. 
Hence (i) holds.   Statement (ii) is an immediate  consequence of (i).  
\end{proof}

  Let $\M$ be a finite algebra.  Then  $\ISP(\M)$, the 
\defn{quasivariety} generated by  $\M$, is the
class of  isomorphic copies of subalgebras of powers of~$\M$.  
It is well known (see \cite[Proposition~2.3]{CP1} for a direct proof)
that $\ISP(\four) = \HSP(\four)$ and that this is the
equational class $\DB$ of distributive bilattices (with bounds).  
Hence $\ISP(\K_0) = \HSP(\K_0)$.
On the other hand, $\HSP(\K_n)$ and $\ISP(\K_n)$ 
do not coincide for any $n\geq 1$.
We have a homomorphism $h_{n,0}$ from $\K_n$ onto $\K_0$,  
and hence $\K_0\in\HSP(\K_n)$. 
Suppose for a contradiction that $\K_0 \in \ISP(\K_n)$.  
Then, since $\K_0$ is not trivial, there exists a homomorphism 
$u\colon \K_0  \to \K_n$. 
Because  $\K_n$  has no proper subalgebras 
no such map~$u$ exists.

Let us fix $n \ge 1$ and consider  $\CV_n= \HSP(\K_n)$.
This variety is lattice-based and hence congruence distributive. 
Therefore we may appeal to J\'onsson's Lemma (see for example 
\cite[Corollary IV-6.10]{BS81}) to assert that every 
subdirectly irreducible algebra in $\CV_n$ is a homomorphic image
 of a subalgebra of $\K_n$ and hence that 
$\HSP(\K_n)= \ISP(\mathbb{HS}(\K_n))$.  
Then, because $\K_n$ has no proper subalgebras, 
$\HSP(\K_n)=\ISP(\mathbb{H}(\K_n))$.  
Corollary~\ref{cor:congKn}  showed that every 
non-trivial congruence in $\Con(\K_n)$ 
arises as the kernel of one of the homomorphisms 
$h_{n,m}$ and,
 moreover, that a non-trivial algebra in 
$\HSP(\K_n)$ is subdirectly irreducible 
if and only if it is isomorphic to $\K_m$ for some 
$m\in \{0,\ldots,n\}$. 
We may now record  the following proposition.  

\begin{prop} \label{Prop:VarKn}
The variety\, $\CV_n $ equals\, 
$\ISP(\{\,\K_m\mid 0\leq m\leq n\,\})$.
The subvarieties of\, $\CV_n$ form an $(n+2)$-element chain 
\[
 \HSP(\K_{-1})\subseteq \HSP(\K_{0})\subseteq \cdots \subseteq \HSP(\K_{n-1})\subseteq\HSP(\K_{n});
\]
here\, $\K_{-1}$ denotes the trivial 
{\upshape(}one-element{\upshape)} bilattice.
\end{prop}

The following lemma is exploited in proving 
that the dualities we present are optimal, in that the dual category 
is as simple as possible. 

\begin{lem}\label{lem:condH} 
Let $m$ and $n$ be such that 
$0\leq m\leq n$.  
Then every homomorphism from\, $\S_{n,m} $ into\, $\K_m$ 
is the restriction of a projection. 
\end{lem}

\begin{proof}  The proof  is a  special instance of a 
classic argument from universal algebra, as given, for example,  
in \cite[Theorem~2.5]{Pix}.  
It uses the fact that $\S_{n,m}$ is lattice-based (and so 
has a distributive  congruence lattice), together with 
 Birkhoff's Subdirect Product Theorem, to show that any 
homomorphic image of $\S_{m,m}$ is a subdirect product of
 homomorphic images of $\K_n$. 
Let $g \colon \S_{n,m} \to \K_m$ be a homomorphism.  
Since $\K_m$ has no proper subalgebras, $g$ is surjective.
Taking account of the fact that $\K_m$ is subdirectly irreducible,  
and has no non-identity endomorphisms, the lemma  follows easily.
\end{proof}

\section{The natural duality framework: multisorted  dualities}\label{sec:multisorted}

Assume that we have a quasivariety  $\CA$ of the form 
$\ISP(\CM)$, where $\CM$ is a finite set of finite algebras, 
later assumed to be lattice-based.  
When 
$\CM$ contains a single algebra~$\M$,
 we write $\CA$ as $\ISP(\M)$. 
We regard $\CA$ as a category, in which the morphisms are all 
homomorphisms. 
 We seek a category $\CX$ of   topological  structures so 
that there are functors $\fnt D \colon \CA \to \CX$ 
and $\fnt E \colon \CX \to \CA$ setting up a dual equivalence.  
This will be done in a very specific way, so that  
$\fnt D $ and $\fnt E$ are given by appropriately defined 
hom-functors.

 An algebra of the same type as those in $\CM$ belongs to $\CA$ 
if and only if the sets of homomorphisms $\CA(\A,\M)$, 
for $\M \in \CM$,  jointly separate the 
elements of~$\A$; 
for an explicit statement and proof of this elementary fact from
 universal algebra, see for example \cite[Theorem~1.1.4]{CD98}.  
This indicates that the hom-sets $\CA(\A,\M)$ may play a role 
in a representation theory for~$\CA$.  
Indeed, Stone duality for $\CB$ (Boolean algebras) and 
Priestley duality for~$\CCD$ (bounded distributive lattices) 
can be seen as capitalising  on this idea: 
each of these classes can be represented as the quasivariety
generated by an algebra~$\M$ with universe $\{0,1\}$.  
One then builds a dual category $\CX$ (of Boolean spaces or  of Priestley spaces, as the case may be).   
There is a natural hom-functor $\D \colon \CA \to \CX$ which, 
on objects, assigns  to $\A$ in $\CA$ the hom-set $\CA(\A,\M)$.
The objects of $\CX$ are obtained by defining an 
\defn{alter ego}~$\MT$ for~$\M$:  
a discretely topologised structure on the same underlying set~$M$. 
For $\CB$, the alter ego $\MT$ is $\{0,1\}$ 
with the discrete topology; for $\CCD$  it is $\{0,1\} $, 
with the partial order~$\leq$ for which $0<1$, 
again with the discrete topology.  
The hom-set  $\CA(\A,\M)$ sits inside  $M^A$, 
equipped with the product topology and, 
in the case of $\CCD$, pointwise lifting of $\leq$.   
The original algebra~$\A$ is recaptured  as the set of continuous
structure-preserving maps from its dual space $\D(\A)$ into
$\MT$,
on which the algebraic operations are defined pointwise
from $\M$.   
Readers familiar with the Stone and Priestley dualities formulated 
in a way different from that we have sketched here can be reassured 
 that passage to the hom-functor approach involves little more
than a simple translation of concepts and notation; for example
 replacement of prime filters by $\{0,1\}$-valued homomorphisms.  
Details can be found in \cite[Chapter~11]{ILO2} and 
\cite[Chapter~1]{CD98}; see also 
Example~\ref{ex:Priestley-duality} below.

Our purpose in outlining the hom-functor perspective on the 
Stone and Priestley dualities has been to provide preliminary
motivation for the multi\-sorted dualities we shall 
employ in this paper.  
 We contend that the ideas involved in setting up the multisorted
 framework  are no more complicated than those in the 
single-sorted case in which $\CM$ contains one algebra only. 
 Accordingly, we shall pass directly to the general case.   
An account which parallels that we give below, but confined to 
 the single-sorted setting, and with distributive bilattices in view,
 can be found in \cite[Section~2]{CP1} 
(see also Example~\ref{ex:dualfour} below).  

The theory of multisorted natural dualities is presented, 
albeit briefly, in \cite[Chapter~7]{CD98}, and in more detail in the
 original source~\cite[Section~2]{DP87}.  
The single-sorted case is much more extensively 
documented than the multisorted one for two reasons.   
Firstly, the former got a head start  and suffices for many
 important applications.  
Secondly, concepts and results in the multisorted setting mimic  
their single-sorted counterparts, so that details have been 
worked out only as potential applications have emerged.    
We shall set up the multisorted duality framework in  detail, 
to make  the constructions  easy to follow.  
But we stress that  it is not necessary to delve into the proofs of 
the general facts we state in order to understand 
the applications we shall make of the results.

Assume we have a quasivariety $\CA = \ISP(\CM)$, where 
$\CM = \{\M_0, \ldots, \M_n\}$  is a set of  non-isomorphic 
finite algebras of common type and having lattice reducts.   
We shall shortly assume in addition that each $\M_i$ 
has  no proper subalgebras and is subdirectly irreducible. 
These assumptions will allow us to  work in a more restricted 
setting than that in \cite{CD98}.
We now need to explain what constitutes an admissible 
\defn{alter ego}  $\CMT$, how the dual category  $\CX$ of
 multisorted structures generated by $\CMT$ is constructed
and how the associated dual adjunction between 
$\CA$ and $\CX$ is set up.

We shall consider an alter ego  for $\CM$ which takes  the form 
\[
\CMT = (M_0 \du \cdots  \du M_n; R, G, \Tp).
\]
Here $R$ is  a  set of relations each of which is a subalgebra 
of some $\M_i \times \M_j$, where $i,j \in \{0,1,\dots,n\}$. 
 If we are to obtain a dual equivalence between $\CA$ and $\CX$
 a purely relational structure may not suffice and  we also allow
 for a set~$G$ of unary algebraic operations. 
By this we mean that each $h \in G$ is a homomorphism 
from  some $\M_i$ into some $\M_j$.
(The reason for assuming that the relations and operations are
 \defn{algebraic} will emerge shortly.)
The alter ego $\CMT$ is given the disjoint union topology derived
 from the discrete topology on the sorts  $M_i$. 

We form multisorted topological 
$\CMT$-structures~${\X = \X_0 \du \cdots \du \X_n} $ 
where each of the sorts $\X_i$ is a Boolean  space,~$\X$ 
is equipped with the disjoint union topology and, 
regarded as a structure, $\X$ carries relations and  operations
 matching those of $\CMT$.  
Thus $\X$ is equipped with a set $R^{\X}$ of relations $r^{\X}$; 
if $r \subseteq M_i \times M_j$, then 
${r^{\X} \subseteq X_i \times  X_j}$;  and similarly $\X$ carries a set 
$G^{\X}$ of unary  
operations.
Clearly $\CMT$ itself is a structure of this type.  
 Given $\CMT$-structures $\X $ and $\Y$, a morphism 
$\phi \colon \X \to \Y$ is  defined to be a continuous map
 preserving the sorts, so that $\phi(X_i) \subseteq Y_i$, 
and  $\phi$ preserves the structure.
The terms isomorphism, embedding, etc., 
are defined in the expected way.

We define our dual category $\CX$ to have as objects those
 $\CMT$-structures $\X$ which belong to the class of  
topological structures which we shall denote by $\IScP(\CMT)$.  
Specifically, $\CX$ consists of isomorphic copies of closed 
substructures of powers of $\CMT$. 
Here powers are formed  `by sorts': given a non-empty set~$S$, 
the underlying set of $\CMT^S$ is the union of disjoint copies
 of~$M^S$, for $\M \in \CM$, equipped with the disjoint union
topology obtained when each $M^S$ is given  the product topology.
The  structure  defined by $R$ and $ G$ is lifted pointwise to 
 substructures of such powers.
The superscript $^+$ indicates that the  empty structure is 
included in $\CX$.

We now define  hom-functors 
$\D \colon\CA \to \CX$ and $\E\colon \CX \to\CA$ 
using 
$\CM$
and its alter ego $\CMT$: 
\begin{alignat*}{2} 
 & \D  \colon \CA \to \CX,  \qquad && \hbox{$
          \begin{cases}
                \D (\A)= \CA(\A,\M_0) \du \cdots \du \CA(\A,\M_n)\\
                \D (f) = - \circ f,
          \end{cases}$}  \\  
 & \E \colon \CX \to \CA,  \qquad && \hbox{$
           \begin{cases}
                 \E (\X)= \CX(\X,\CMT)\\
                 \E (\phi)  = - \circ \phi .
           \end{cases}$}  
\end{alignat*}
Here the disjoint union 
$\CA(\A,\M_0) \du \cdots \du \CA(\A,\M_n)$ is  a 
(necessarily closed) substructure of $M_0^A\du \cdots \du M_n^A$
 and so a member of $\IScP(\CMT)$. 
 We recall from above that  $\CX(\X,\CMT)$, as a set, 
is the collection of continuous structure-preserving maps
 $\phi\colon \X \to\CMT$ which are such that 
$\phi(X_i)\subseteq M_i $ for $0 \leq i \leq n$.
This set acquires the structure of a member of 
$\CA$ by virtue of viewing it as a subalgebra of the power 
$\M_0^{X_0} \times  \dots \times \M_n^{X_n}$. 
The well-definedness of the functors $\D$ and $\E$ is of central 
importance to our enterprise. 
It hinges on the assumption we have made that the relations and 
operations in the alter ego are algebraic, and that each $M_i$ is 
finite and carries the discrete topology; 
cf. \cite[Preduality Theorem, 2.5.2]{CD98}.    
We can say more (cf.~\cite[Dual Adjunction Theorem, 2.5.3]{CD98}): 
 $\D$ and $\E$ set up  a  dual adjunction, $(\D,\E, e,\varepsilon)$ 
in which the unit and counit maps are evaluation maps,  
and these evaluations  are embeddings.

We say $\CMT$
 \defn{yields a multisorted duality} if, for each 
$\A \in \CA$,  the evaluation map $e_{\A}\colon
\A \to \ED(\A)$
is an isomorphism.  
 The duality is \defn{full} if, for each $\X \in\CX$, the evaluation map 
$\varepsilon_{\X} \colon \X \to \D\E(\X)$ is an isomorphism.  
Thus a duality provides a concrete 
representation $\E\D(\A)$ of  $\A \in \CA$.
If in addition the duality is full, we also know that every 
$\X \in \CX$ arises, up to isomorphism,  as a topological structure 
$\D(\A)$, for some $\A \in \CA$.

In practice, fullness of a duality  is normally obtained at 
second hand by showing that the duality is  \defn{strong}. 
We do not need to use this notion directly; for the formal 
definition see \cite[Chapter~3]{CD98}.  
  However we do remark  that the functors $\D$ and $\E$ 
setting up a strong duality have  the property that 
each maps 
an embedding to a surjection and a surjection to an embedding; 
this is a very desirable feature of a duality as regards applications. 

We record  an important fact, true for  any multisorted duality, 
and adding weight to the duality's claim to be called `natural'.  
In $\CA = \ISP(\CM)$, the free algebra $\Free_{\CA}(S)$  on a set $S$ 
of generators is isomorphic to $\E(\CMT^S)$;  
in particular, $\CMT$ is the dual space of $\Free_{\CA}(1)$
 \cite[Lemma~2.2.1 and Section~7.1]{CD98}.

We  now state, without further ado, the theorem on which we shall
rely, following it with an informal commentary.
It is a  very restricted form of \cite[Theorem~7.1.2]{CD98}  
which draws also,  {\it mutatis mutandis}, 
on \cite[Corollary~3.3.9]{CD98}.

\begin{thm}  
{\rm(Multisorted NU Strong Duality Theorem, special case)}\label{thm:NUnSorts}
Let\, $\CA = \ISP(\CM)$, where\, $\CM = \{ \M_0, \dots,\M_n\}$ 
is a set of non-isomorphic subdirectly irreducible   
algebras of common type having lattice reducts
and assume  that no\, $\M_i$ has a proper subalgebra. 
Let\, $\CMT = \bigl(\,
M_0 \du \cdots\du  M_n;  
R, G,\Tp \bigr) $
where $ R= \bigcup 
\{\, \mathbb{S} (\M_i\times\M_j)\mid i,j\in\{0,1,\dots,n\}\,\}$,
$G= \bigcup 
\{\, \CA (\M_i\times\M_j)\mid i,j\in\{0,1,\dots,n\}\,\}$,
 and $\Tp$ is the disjoint union topology obtained from  
the discrete topology on the sorts $M_i$.  
Then\, $\CMT$ yields a multisorted duality on\, $\CA$ which is
 strong {\upshape(}and hence full\,{\upshape)}.
\end{thm}

It is clear that, from the perspective of universal algebra, 
the restrictions we have imposed on $\CM$ are extremely stringent.  
However the results of the previous section show that all 
the assumptions are met when $\CM = \{\,\K_0, \dots, \K_n\,\}$ 
(for any $n \geq 0$).  
We could also  take $\CM = \{ \K_n\}$, to obtain a single-sorted
 duality for the quasivariety $\ISP(\K_n)$; 
see Section~\ref{sec:dualitiesKn}.

The assumption that the 
algebras 
in $\CM$ be lattice-based comes
 into play in the following way.  
Since each  $\M\in\CM$ has a lattice reduct, it has a 
$3$-ary near unanimity term, {\it viz.}~the lattice median.   
This ensures, as a consequence of the multisorted version of the NU
Duality Theorem~\cite[Theorem~3.3.8 and Corollary~3.3.9]{CD98}  
that the set of all binary  relations which are subuniverses of 
algebras $\M_i \times \M_j$ (where $i,j$ vary over $\{0,\dots,n\}$)
 yields a duality on $\CA= \ISP(\CM)$. 
 We stress that a critical part of the conclusion here is dualisability:  
there exists an alter ego yielding a duality.  
Moreover we obtain  as a bonus  a very explicit 
form of one such alter ego.

Finally we should comment on the claim in the theorem 
that the duality is strong.  
A duality can fail to be full if the alter ego is insufficiently rich, 
so that the dual category $\CX$ is too big.   
 In the lattice-based case, adding additional structure to the 
alter ego in the form of algebraic operations (sometimes partial),
 one can arrive at a duality which is strong, and hence full.  
However, under the very restricted conditions imposed on  $\CM$ in 
Theorem~\ref{thm:NUnSorts},  it turns out that unary total 
operations suffice.

We would like the dualities we present to contain in their alter egos
 as few relations and operations as possible. 
 Suppose  $\CMT$, as in Theorem~\ref{thm:NUnSorts},  is an alter
 ego  yielding a duality on  a class $\CA=\ISP(\CM)$.
Then any $\A \in \CA$ is such that $\A \cong \E \D(\A)$;  
here the structure of $\D(\A) $ is completely  determined by that of
 $\CMT$,  and the elements of $\E\D(\A)$ are the multisorted 
continuous structure-preserving maps from $\D(\A)$ into $\CMT$. 
 From this it is clear that, for example, we gain nothing by 
including in $R$ both a binary relation and its converse.  
 It is also never necessary to
 include `trivial relations': those which are preserved automatically 
by $\CX$-morphisms.    
Examples are $\M_i^2$ and its diagonal subalgebra, for any~$i$.
Here we have very simple 
instances of \defn{entailment},
sufficient for our immediate needs;
see \cite[Section 2.4]{CD98} for further information.

We conclude this summary of facts from natural duality theory 
by drawing attention to two (single-sorted) dualities   which  
fit into the special framework we have described and can be derived,
albeit circuitously,  from Theorem~\ref{thm:NUnSorts} and 
the remarks above. 
First we revisit Priestley duality, which we mentioned briefly at 
the start of this section.
This provides  a valuable tool for working with $\CCD$-based
 algebras,
on which we shall draw heavily in Section~\ref{sec:RevEng}:
recall that the knowledge lattice reduct of each algebra in 
$\CV_n$  belongs to $\CCD$.

\begin{ex}  \label{ex:Priestley-duality}  
(Priestley duality)
We recall that a \emph{Priestley space} is a topological structure 
$(X; \leq,\Tp)$ in which $(X;\Tp)$ is a compact space and
$\leq$ is a partial order with the property that, given 
$x \nleqslant y$ in~$X$, there exists a $\Tp$-clopen up-set~$U$
such that $x \in U$ and $y \notin U$.  
The morphisms in the category~$\CP$ of Priestley spaces are 
the continuous order-preserving maps.

There is a dual equivalence between 
$\CCD$ and $\CP$ constructed as follows.  
Let $\two= \bigl( \{ 0,1\}; \land,\lor,0,1\bigr)$ 
be the two-element lattice in $\CCD$ and let its alter ego be
$\twoT = \bigl(  \{ 0,1\};  \leq, \Tp\bigr)$.
Then $\CCD = \ISP(\two)$ and $\CP = \IScP (\twoT)$
and the  hom-functors  $\fnt{H} = \CCD(-, \two)$ and 
$\fnt{K} = \CP(-, \twoT)$ set up a strong duality
between~$\CCD$ and~$\CP$.   
This can be seen as a consequence of the single-sorted case of
 Theorem~\ref{thm:NUnSorts} and our comments on entailment.  
Since we  later use Priestley duality in conjunction with a
 natural duality based on hom-functors $\D$ and~$\E$, 
we adopt non-generic symbols  
$\fnt{H} $ and $\fnt{K}$ for the hom-functors 
between~$\CCD$ and~$\CP$.   
\end{ex}

Our second example serves to indicate that, for  the base case of
 our hierarchy of default bilattices,  a natural duality has 
 already been worked out. 
We present this as 
for 
$\ISP(\K_0)$, recalling that this variety  is 
term-equivalent to the variety $\DB$ studied in \cite{CP1}.

\begin{ex} \label{ex:dualfour}
 (Natural duality for distributive bilattices 
\cite[Theorem 4.2]{CP1})
There is a dual equivalence between the category $\ISP(\K_0)$ 
of distributive bilattices and the category $\CP$ of 
Priestley spaces constructed as follows.  
The alter ego 
$
\twiddle{\spc{K}_0} = \bigl(  \{ \top,\bot, \mbt_0, \mbf_0\};  \leq_k, \Tp\bigr)
$ 
for  $\K_0$ yields a strong (and hence full) duality on 
$\CV_0 = \ISP(\K_0)$ and moreover
the dual category $\IScP (\twiddle{\spc{K}_0})$ 
coincides
with $\CP$.
We do not justify here the identification of the dual category, 
which can be found in \cite{CP1}.  
We do, however draw attention to the occurrence of 
the knowledge order $\leq_k$ in the alter ego.  
Because $\K_0$ is a distributive bilattice, $\leq_k$ is an 
algebraic relation.  
\end{ex}  

\section{Dualities for varieties of prioritised default bilattices}
\label{sec:varKn}

In this section we present multisorted natural dualities 
for the varieties $\CV_n=\HSP(\K_n)$, for $n \geq 1$.  
 (The strategy  we use applies equally to the case $n=0$, 
but we have already commented on this simple case,  in which 
 the duality is  single-sorted because $\CV_0 = \ISP(\K_0)$.) 

We shall apply the Multisorted NU Strong Duality Theorem in the 
restricted form stated as Theorem~\ref{thm:NUnSorts} and  then  
use entailment arguments to simplify the alter ego.
 We refer the uninitiated but interested reader   to 
\cite[Section~9.4]{CD98} for  an account of entailment as it
 applies to strong dualities and to \cite[Section~9.2]{CD98}
for definitions of the entailment constructions we invoke.

\begin{thm}\label{Thm:DualityVarietyKn}
 Consider\, 
$\CV_n 
= \ISP(\CM_n)$,  where\, 
$\CM_n =\{\K_{0},\ldots, \K_{n} \}$.
Then the alter ego
\[\twiddle{ \CM_n} = ( K_{0} \du \cdots \du K_{n} ; R_n,G_n,\Tp ), 
\quad\mbox{where } 
R_n=\{\, S_{m,m}\mid 0\leq m\leq n\}  
\mbox{ and } \,   
G_n=\{\,h_{i,i-1}\mid 1\leq i\leq n\,\}
\]
yields a strong {\upshape(}and hence full\,{\upshape)} 
duality on\, $\CV_n$.  
\end{thm}
\begin{proof} 
The representation of $\CV_n$ as $\ISP(\CM_n)$ was established in  
Section~\ref{sec:ClassKn}, where 
we also proved that each  
$\K_m$ is subdirectly irreducible and has no proper subalgebras.
Hence the structure
$\twiddle{\CM_n}' = ( K_{0}\du \cdots \du K_{n} ; 
R, G, \Tp)$, 
where $R=\bigcup\, 
\{\,\mathbb{S}(\K_{i}\times \K_{j})\mid i,j \in \{0,\ldots,n\}\,\}$
 and 
$G=\bigcup\,
\{\,\CV_n(\K_{i},\K_{j})\mid i,j \in \{0,\ldots,n\} \,\}$,
yields a strong duality on $\CV_n$. 
The theorem is then a consequence of two claims.

\begin{claimslist}
\item[{\bf Claim 1}:]
$G_n$ hom-entails $G$ 
(in the sense of the definition in \cite[Section~3.2]{CD98}).

By Corollary~\ref{cor:congKn}, $\CV_n(\K_{i},\K_{j})=\{h_{i,j}\}$ 
if $j\leq i$ and $\CV_n(\K_{i}, \K_{j})=\emptyset$ otherwise. 
If $i=j$ then $h_{i,j}$ is the identity on $\K_{i}$ and  
it is straightforward to check that  
$h_{i,j}=h_{j+1,j}\circ\cdots \circ h_{i,i-1}$ when $j<i$. 

\item[{\bf Claim 2}:]
$R_n$ and $G_n$ entail~$R$.

From Proposition~\ref{prop:Snn-subalg}(iv),  
$\mathbb{S}(\K_{i}\times\K_{j})=
\{(h_{n,i}\times h_{n,j})(\S)\colon \S\in\mathbb{S}(\K_n^2)\}$.
Using the entailment constructs in  \cite[Section 9.2]{CD98} 
(specifically term manipulation and 
homomorphic relational product), 
we see that  $\mathbb{S}(\K_n^2)$ and 
$G$
together entail~$R$. 
By Claim~1, $\mathbb{S}(\K_n^2)$ and $G_n$ entail $G$. 
Also, from Proposition~\ref{prop:Snn-subalg}(i) we have $S_{n,m} =\{\,(a,b)\mid (h_{n,m}(a),h_{n,m}(b))\in S_{m,m} \,\}$.  
We deduce from Theorem~\ref{thm:Snmonly} and 
\cite[Section 2.4]{CD98} that $\mathbb{S}(\K_n^2)$ is entailed by
 $R_n$ and $G_n$. \qedhere
\end{claimslist}
\end{proof}

Referring to the proof of Theorem~\ref{Thm:DualityVarietyKn}, 
let us see what simplification of the duality for $\CV_n$ is achieved 
by replacing $G$ and $R$ by $G_n$ and $R_n$. 
We have $|G|=\frac12{(n+2)(n+1)}$. By contrast, $|G_n|=n$. 
Let $B_1=\bigcup\{\,\mathbb{S}(\K_i^2) \mid 0 \le i \le n \,\}$. 
From Theorem~\ref{thm:Snmonly}, 
$|B_1|=
 \frac12{(3n+8)(n+1)}$. 
We use Proposition~\ref{prop:Snn-subalg}(iv) to calculate 
$|B_2|$ where $B_2=R\setminus B_1$.
We obtain 
$|B_2|=\sum_{j=0}^{n-1} 2(n-j)(3j+4)$. 
Hence $|R|=\tfrac12{(3n+8)(n+1)} 
+ n(n+1)(n+3)
=\tfrac12{(n+1)(2n^2+9n+8)}$
whereas $|R_n| =n+1$.  
But can we, using binary relations and unary operations, 
do any better?

In the setting of Section~\ref{sec:multisorted},
$\CMT = ( M_0 \du \dots \du M_n; R, G, \Tp)$ is said to yield an
\defn{optimal duality} on $\CA=\ISP(\CMT)$ if $\CMT$ yields 
a duality on $\CA$ and if the alter ego $\CMT'$ obtained by deleting 
any  member of  $R \cup G$ fails to do so.   
Recalling that each $r \in R$ is assumed to be algebraic, 
and so the subuniverse of an algebra $\mathbf{r}$ 
in $\CA$, we may  use  $\mathbf{r}$ as a test algebra  and  seek to 
show that it is not true that 
$\mathbf{r} \cong \D '\E' (\mathbf{r})$; here $\D ' $ and $\E '$ 
denote   the functors associated with the alter ego $\CMT'$ 
obtained by deleting~$r$ from $\CMT$.  
Likewise,  we may seek to show   an element $h \in G$ 
is not redundant  by using  $\dom h$ as a test algebra.  
(For a full account of the test algebra 
strategy, see \cite[Section~8.8.1]{CD98}.) 

\begin{thm}\label{thm:optimalstrong} 
The alter ego\, 
$\CMT_n=(K_0 \du\cdots\du K_n;R_n,G_n,\Tp)$, 
as defined   
 in Theorem~{\upshape\ref{Thm:DualityVarietyKn}},
yields an optimal duality on\,~$\CV_n$.
\end{thm}
\begin{proof} 
Recall that relations and homomorphisms are lifted from the multisorted alter ego 
pointwise, by sorts.
In particular, if $\A \in \CV_n$
then, for a binary relation~$r=S_{j,j} $, and $x,y \in \D(\A)$,
\[
(x,y) \in r^{\D(\A)}  \Longleftrightarrow \forall a \in \A
\bigl(  
(x(a),y(a)) \in S_{j,j}\bigr).
\]
For this to hold,  necessarily  $x,y\in \CV_n(\A,\K_j)$.  Hence 
$r^{\D(\A)}$ is the empty relation whenever 
$\CV_n(\A,\K_j) = \emptyset$.
 Likewise,  for a homomorphism $h_{i,i-1}$, 
\[
y= h_{i,i-1}^{\D(\A)}(x)  \Longleftrightarrow \forall a \in \A
\bigl(  y(a) = h_{i,i-1}(x(a)) \bigr)
\Longleftrightarrow 
y = h_{i,i-1} \circ x,
\] 
and for this to hold it is necessary that $x \in \CV_n(\A,\K_i)$ and 
$y \in \CV_n(\A,\K_{i-1})$.

Fix $m$. We will show that if $S_{m,m}$ were deleted from $R_n$, 
then the map $e_{\S_{m,m}} \colon \S_{m,m} \to \ED(\S_{m,m})$ 
cannot be surjective.  
Recall that 
$\D(\S_{m,m})= \dubig \{\, \CV_n(\S_{m,m},\K_j) \mid j \in \{0,\ldots,n\}\,\}$.
To simplify notation, we shall denote $\D(\S_{m,m})$ by~$\X$.
By Lemma~\ref{lem:condH}, 
$\CV_n(\S_{m,m},\K_m)=\{\rho_1,\rho_2\}$, where $\rho_1$ and $\rho_2$
are the restrictions of the coordinate projections. 
Also,
 $\CV_n(\S_{m,m},\K_j)=\emptyset$ if $j>m$. Now consider  $j<m$. 
Any homomorphism $g \colon \S_{m,m} \to \K_j$ is such that
$g(\top_{\!\!j+1},\top_{\!\!j+1})=\top_{\!\!j+1}$
and hence for all 
$(a,b) \le_k (\top_{\!\!j+1},\top_{\!\!j+1})$  we have 
$g(a,b)=\top_{\!\!j+1}$. 
This implies that 
$g=h_{m,j}\circ \rho_1=h_{m,j}\circ \rho_2$ and 
hence $\CV_n(\S_{m,m},\K_j)=\{h_{m,j}\circ \rho_1\}$.
Therefore 
$\X=
\{\rho_1,\rho_2\} \cup \{\, h_{m,j}\circ\rho_1 \mid 0\le j<m\,\}$.
It is easy to check that
$(\rho_1,\rho_1)$, $(\rho_2,\rho_2)$ and $(\rho_1,\rho_2)$ belong 
to $S_{m,m}^{\X}$ whereas $(\rho_2,\rho_1)$ does not. 
We now want to construct a map $\gamma_m \colon \X\to\CMT_n$ 
such that $\gamma_m$ preserves each member of 
$(R_n \setminus\{S_{m,m}\} )\cup G_n$ but does not preserve 
$S_{m,m}$. 
 We define $\gamma_m$ by 
\[
\gamma_m(x)= 
  \begin{cases} 
      \mbf_{m} &\text{if } x=\rho_1, \\
      \mbt_m &\text{if } x=\rho_2, \\
      \top_{\!\!j+1} &\text{if } x=h_{m,j}\circ \rho_1. 
  \end{cases}
\]
This map 
does not preserve $S_{m,m}$ as 
$(\gamma_m(\rho_1), \gamma_m(\rho_2))\!=
\!(\mbf_m,\mbt_m) \notin S_{m,m}$. 
Consequently $\gamma_m$ cannot be an evaluation map. 
Now we wish to show that $\gamma_m$ does preserve
the remaining structure in $\CMT_n$. 
First we deal with the relations $S_{j,j}$ for  $j \ne m$. 
If $j>m$ the relation $S_{j,j}^{\X}$ is empty.
Now consider $j<m$. 
The only element in $\CV_n(\S_{m,m},\K_j)$ is $h_{m,j}\circ\rho_1$.
 Since $S_{j,j}$ is reflexive, ${S_{j,j}^{\X} = 
\{(h_{m,j}\circ\rho_1,h_{m,j}\circ\rho_1)\}}$.
Hence $\gamma_m$ preserves  $S_{j,j}$ whenever $j \ne m$.
 We claim also that $\gamma_m$ preserves  $h_{i,i-1}$. 
If $i>m$ then 
$\CV_n(\S_{m,m},\K_i)=\emptyset$ 
and  so 
$h_{i,i-1}^{\X}$ is the empty map and trivially preserved.
If $i\leq m$ then $\CA(\S_{m,m},\K_{i-1})=
\{h_{m,i-1}\circ\rho_1\}$. 
Thus $h_{i,i-1}^{\X} = h_{m,i-1}\circ\rho_1$. 
So, for  $i<m$,
\[
\gamma_m\bigl(h_{i,i-1}^{\X}(h_{m,i}\circ\rho_1)\bigr)=
\top_{\!\!i}=h_{i,i-1}(\top_{\!\!i+1})=
h_{i,i-1}(\gamma_m(h_{m,i}\circ\rho_1)).
\]
If  $i=m$  we have 
$\gamma_m\bigl(h_{m,m-1}^{\X}(\rho_1)\bigr)=
\top_{\!\!m-1}=h_{m,m-1}(\mbf_{m})=
h_{m,m-1}(\gamma_m(\rho_1))$, 
and likewise 
$\gamma_m\bigl(h_{m,m-1}^{\X}(\rho_2)\bigr)=
h_{m,m-1}(\gamma_m(\rho_2))$. 
We conclude that $\gamma_m$ preserves $h_{i,i-1}$ for each 
$i\leq m$.

Now  we show that,  for $1\leq m\leq n$,  we cannot remove the 
 homomorphism $h_{m,m-1}$ from the alter ego.
We do this by considering  $\Y =\D(\K_m)$.
We have $\CV_n(\K_m,\K_i)=\emptyset$
if  $i>m$ and $\CV_n(\K_m,\K_i)=\{h_{m,i}\}$ if  $i\leq m$.
Therefore if $i>m$ then $h_{i,i-1}^{\Y}$ is the empty map.
If $i \leq m$ then $h_{i,i-1}^{\Y}(h_{m,i})$ has domain $\K_m$
 and is the map $h_{i,i-1} \circ h_{m,i} = h_{m,i-1}$. 
Moreover, $S_{i,i}^{\Y}=\emptyset$ if $i>m$ and $S_{i,i}^{\Y}=\{(h_{m,i},h_{m,i})\}$ for $i\leq m$.
Define $\mu_m\colon \Y \to \CMT_n$ by 
\[
\mu_m(h_{m,i})= 
\begin{cases} 
  \top_{\!\!m-1} &\text{if } i=m, \\
  \top_{\!\!i+1} &\text{if } i<m. 
\end{cases}
\]
Trivially,  $\mu_m$ preserves each $S_{j,j}$.   
Moreover, if $i<m$ then 
\[
\mu_m(h_{i,i-1}^{\Y}(h_{m,i}))=\mu_m(h_{m,i-1})=
\top_{\!\!i}=h_{i,i-1}(\top_{\!\!i+1})=h_{i,i-1}(\mu_m(h_{m,i})),
\]
that is, $\mu_m$ respects $h_{i,i-1}$ for any  $i<m$. But
$\mu_m$ does not respect   $h_{m,m-1}$:
\[
\mu_m(h_{m,m-1}^{\Y}(h_{m,m}))=\mu_m(h_{m,m-1})=\top_{\!\!m} 
\quad \mbox{  and } \quad
h_{i,i-1}(\mu_m(h_{m,m}))=h_{m,m-1}(\top_{\!\!m-1})= \top_{\!\!m-1}. \hspace*{1.8cm} \qedhere
\]
\end{proof}

We shall now characterise the objects in our dual category $\IScP(\twiddle{\CM_n})$.

\begin{thm}\label{Thm:CharDualVn}
Let $\twiddle{\CM_n}$ be the alter ego  
 defined in Theorem~{\upshape \ref{Thm:DualityVarietyKn}}. 
Then a multisorted topological structure
\[
\X=(X_0\du \ldots \du X_n ;
\leq_0,\ldots,\leq_n,g_1,\ldots, g_{n},\Tp),
\] 
where $\leq_i\, \subseteq X_i^2$ for $0\leq i\leq n$ and 
$g_j\colon X_j\to X_{j-1}$ for $1\leq j\leq n$,  belongs to\, 
$\IScP(\twiddle{\CM_n})$, 
 if and only if
\begin{newlist}
\item[{\upshape(i)}]  $(X_i;\leq_i,\Tp_i)$ is a Priestley space for 
$i\in\{0,\ldots, n\}$, where $\Tp_i$ 
is the topology induced by $\Tp$; 
\item[{\upshape(ii)}] $g_i\colon(X_i;\Tp_i) \to (X_{i-1};\Tp_{i-1}) $ 
is continuous, for $i\in\{1,\ldots,n\}$;
\item[{\upshape(iii)}] 
if $x\leq_i y$ then $g_i(x)=g_i(y)$, for 
$i\in\{1,\ldots,n\}$.
\end{newlist}
\end{thm}
\begin{proof}
Clearly $S_{m,m}$ is a partial order on $K_{m}$.
And if $a,b\in\K_m$ with $1\leq m\leq n$ are such that $a\neq b$ 
and $(a,b)\in S_{m,m}$, then $a,b\leq \top_{\!\!m}$. 
Then $h_{m,m-1}(a)=h_{m,m-1}(b)=\top_{\!\!m}$. 
This proves that  $\twiddle{\CM_n}$ satisfies (i), (ii) and (iii).
 Since each of (i), 
(ii) and (iii) are preserved under products and 
closed substructures, each 
$\X\in\IScP(\twiddle{\spc{K}_n})$ 
satisfies them.
	
To prove the converse, we shall invoke \cite[Theorem~1.4.4]{CD98}.
 Assume that  $\X$ satisfies (i), (ii) and (iii) and let $x,y\in X$ and 
$i\in\{0,\ldots,n\}$ be such that $x\not\leq_i y$.
By (i), since $(X_i;\leq_i,\Tp_i)$ is a Priestley space, 
 there exists a 
clopen up-set $U_{x,y,i}$ such that $x\in U_{x,y,i}$ and 
$y\notin U_{x,y,i}$.
Define $U_{j}\subseteq X_j$ by 
\[
U_j=
\begin{cases}
   U_{x,y,i}&\mbox{ if }j=i,\\
   X_j&\mbox{ if }j<i,\\
   (g_{j}\circ\cdots\circ g_{i+1})^{-1}(U_{x,y,i})&\mbox{ if }j>i.  
\end{cases}
\]
By (ii), each $g_j$ is continuous and,  since $U_{x,y,i}$ is clopen,
each $U_{j}$ is clopen. 
By (iii), each $U_j$ is also an up-set. 
Define  $f_{x,y,i}\colon\X \to \twiddle{\CM}$  by letting
$f_{x,y,i}(z)=\top_{\!\!i}\in K_{j}$  if $z\in U_{j}$ and 
$f_{x,y,i}(z)=\top_{\!\!j+1}\in  K_{j}$ otherwise.
Each  $U_{j}$ is a clopen up-set, so  $f_{x,y,i}$ is order-preserving
and continuous sort-wise. 
Let $z\in X_j$ with $1\leq j\leq n$. 
If $j<i$ then $f_{x,y,i}(z)=\top_{\!\!j+1}$ and 
$f_{x,y,i}(g_{j}(z))=\top_{\!\!j}$. 
Then $f_{x,y,i}(g_{j}(z))=g_{j}(f_{x,y,i}(z))$. 
If $j> i$, then $f_{x,y,i}(g_{j}(z))=\top_{\!\!i}$ if and only if 
$g_{j}(z)\in U_{j-1}$. 
That is,  $f_{x,y,i}(g_{j}(z))=\top_{\!\!i}$ if and only if $z\in U_{j}$. 
Then  $f_{x,y,i}(g_{j}(z))=\top_{\!\!i}$ if and only if 
$f_{x,y,i}(z)=\top_{\!\!i}$. 
We deduce that $f_{x,y,i}(g_{j}(z))=g_{j}(f_{x,y,i}(z))$. 
If $j=i$, then $g_j(f_{x,y,i}(z))=g_i(f_{x,y,i}(z))=
\top_{\!\!i}=f_{x,y,i}(g_i(z))$.
This proves that $f_{x,y,i}$ is a morphism from $\X$ into 
$\twiddle{\CM_n}$ and that 
$(f_{x,y,i}(x), f_{x,y,i}(y))=(\top_{\!\!i} ,\top_{\!\!i+1})\notin S_{i,i}$. 
It follows that $\X\in\IScP(\twiddle{\CM_n})$.
\end{proof}

\section{Relating the natural duality for  
$\CV_n$ to Priestley 
duality}
\label{sec:RevEng} 

The principal result in this section is Theorem~\ref{Thm:RevEngVarietyKn}.
It will enable us to give 
information about free algebras in the varieties
$\CV_n$ and will later throw light on the  product 
representation we present in Section~\ref{sec:ProdRep}.  
However in carrying out our analysis 
we call on recent results from \cite{CP1}, and so on aspects of 
duality theory for $\CCD$-based algebras that we have not needed 
hitherto.  
Section~\ref{sec:ProdRep} can if desired be read
with almost no reference to this section.

We shall  presuppose that the reader has some familiarity with 
basic facts concerning  Priestley duality and its 
consequences.
 We 
recall that we use the non-generic symbols $\fnt{H}$ 
and 
$\fnt{K}$ for the functors setting up Priestley duality,
retaining $\D$ and $\E$ for the functors setting up the 
 duality for $\CV_n$ given in 
Theorem~\ref{Thm:DualityVarietyKn}.
We may  identify a lattice~$\Lalg$ in $\CCD$ with $\fnt{KH}(\Lalg)$. 
Identifying a continuous order-preserving function $x$ from  
$\fnt{H}(\Lalg)$, the Priestley dual space of $\Lalg$,
into $\twoT$ with the set $x^{-1}(1)$, we may when convenient 
regard $\Lalg$ as the lattice of clopen up-sets of $\fnt{H}(\Lalg)$.   
A full account of Priestley duality and its consequences  can be 
found for example in \cite[Chapters~5 and 11]{ILO2}, 
but  we warn that the treatment there works with down-sets 
rather than up-sets.

As we have observed earlier,   the algebras in $\CV_n$ have reducts in the
 category $\CCD$ of  bounded distributive lattices.  
Formally, there exists a natural forgetful 
functor $\fnt{U} $ from  $\bigcup_{n\geq0} \, \CV_n $ to $\CCD$, 
sending an algebra  $\A$ to $(A;\otimes,\oplus,\bot,\top)$  
and each morphism to the same map, now  regarded as a 
$\CCD$-morphism. 
We shall investigate the relationship between 
the natural duality we have set up for $\CV_n$ on the 
one hand and Priestley duality as it applies to the subcategory 
$\fnt{U}(\CV_n)$ of $\CCD$ on the other.  
Of necessity, we work 
with knowledge lattice reducts since for $n \geq 1$ the truth 
lattice reduct is not distributive.  
This means that the treatment below  
does not align fully 
with that for $n=0$ given in \cite{CP1}.  
The difference is more notational than real and 
we can recommend the account given in \cite{CP1} for the special case 
as an introduction to ideas we shall use  also for general~$n$. 

Fix $n \geq 1$.  We want to know how the  multisorted dual space 
$\D(\A)$ is related to  the Priestley dual space  $\fnt{HU}(\A)$ of 
$\fnt{U}(\A)$ for $\A\in \CV_n$ 
(from which $\fnt{U}(\A)$ can be recovered by Priestley duality).
For any  finitely generated $\CCD$-based variety,  and in particular 
for $\CV_n$, it is possible to set up an economical natural duality 
by  what is known as the piggybacking method,
 without
recourse to the NU Duality Theorem; 
see~\cite[Chapter~7]{CD98}. 
Furthermore,  this piggyback duality can be related to Priestley 
duality as it applies to the $\CCD$-reducts,
as shown in 
\cite[Section~2]{CPcop}.
 We opted, however,  not to employ this method to set up a natural 
duality for $\CV_n$. 
To have done so  would have involved at the outset additional  
theoretical machinery  and would not have yielded a quicker 
or more informative  derivation.  
But now, with the insights gleaned from the approach we adopted in 
Sections~\ref{sec:ClassKn}--\ref{sec:varKn}, it is profitable to 
reconcile Theorem~\ref{Thm:DualityVarietyKn} with 
results from \cite{CPcop}.  
This reconciliation  elucidates Theorem~\ref{Thm:CharDualVn}
and provides a bridge to the  
product representation in 
due course.

In preparation for Theorem~\ref{Thm:RevEngVarietyKn} we need to
relate our duality for $\CV_n$ from 
Theorem~\ref{Thm:DualityVarietyKn}  to the results of
 \cite[Section~2]{CPcop} as they apply to $\CA$, 
where $\CA= \CV_n = \ISP(\CMT_n)$.
The key here---and we cannot emphasise this too strongly---%
is the relationship between $\CMT_n$, viewed as a member of our 
natural dual category, and the sets $ \CCD(\fnt{U}(\K_i), \two)$,  
for $0 \leq i \leq n$.  
For each such~$i$ let $\omega_i^t $ and $\omega_i^f $ be the 
elements of   $\CCD(\fnt{U}(\K_i), \two)$ for which 
$(\omega_i^t)^{-1}(1) =  {\uparrow} \mbt_i$ and
 $(\omega_i^f)^{-1}(1) = {\uparrow} \mbf_i$;
 here the up-sets are calculated with respect to the 
$\leq_k$ order on $K_i$.   
Let $\Omega_{\CMT_n} = 
\bigcup_{0\leq i \leq n} \, \{\omega_i^t, \omega_i^f\}$.

\begin{lem}  \label{lem:piggy-reconcile}
Let $\Omega_{\CMT_n}$ be as above.  
\begin{newlist}
\item[{\upshape (i)}] 
The following separation condition holds:
given $m\in \{0, \ldots, n\}$ and $a \ne b \in \K_m$, there exists 
$j \in \{0,\ldots,m\}$  and 
$\omega \in  \{\omega_j^t,\omega_j^f\}$ 
such that  $\omega (h_{m,j}(a)) \ne \omega (h_{m,j}(b))$.

\item[{\upshape (ii)}]  Let $j,m \in \{ 0,\ldots,n\}$.   For 
$\omega\in \{ \omega_j^t,\omega_j^f\}$ and $\omega'  \in \{
\omega_m^t,\omega_m^f\}$, let $R_{\omega,\omega' }$  
be the set of binary algebraic relations which are 
 maximal with respect to being contained in
$
\{\, (a,b) \in K_j \times K_m \mid  \omega(a) \leq \omega'  (b)\,\}$.
Then 
\begin{newlist}
\item[{\upshape (a)}]  $R_{\omega_m^t,\omega_m^t} = 
R_{\omega_m^f,\omega_m^f}  =\{ S_{m,m}\}$;
 \item[{\upshape (b)}]  if $j = m$ and $\omega \ne \omega' $ or 
$j > m$, then $R_{\omega,\omega' }= \emptyset$;
\item[{\upshape (c)}]  if $j < m$ then 
$R_{\omega,\omega' } = \{ 
(\graph h_{m,j})^\smallsmile\}$.
\end{newlist}
\end{newlist} 
\end{lem}
\begin{proof}  
Consider  (i). 
Take $a \ne b$ in $\K_m$, and assume without loss of generality 
that $a \nleqslant_k b$. We consider two cases. 
Assume first  that  there exists $j\in\{0,\ldots,n\}$ with 
$b<_k \top_{\!\!j+1} \leq_k a$. 
Then $\omega_j^t(h_{m,j+1}(a)) = \omega_j^f(h_{m,j+1}(a)) = 1$.  
On the other hand, $\omega_j^t(h_{m,j}(b))= 0$ or 
$\omega_j^f(h_{m,j}(b)) = 0$, or both,  must hold.
Now  assume that 
$a,b \in  \{ \top_{\!\!j}, \mbt_j, \mbf_j, \top_{\!\!j+1} \}$
for some~$j$.  
Then for $\omega = \omega_j^t$ or for $\omega = \omega_j^f$ 
we have $\omega(h_{m,j}(a)) \ne \omega(h_{m,j}(b))$.   
    
All  the ingredients for the proof  of (ii) are  given  
in our 
earlier analysis of subalgebras of products $\K_j \times \K_m$. 
We know that $S_{m,j} \subsetneq S_{m,i}$ for 
$0 \leq i < j \leq m$.  This fact, combined with Theorem~\ref{thm:Snmonly},
tells us that (ii)(a) holds provided  
$\omega_m^t(a) \leq \omega_m^t(b)$ for each $(a,b)\in S_{m,m}$ and $\omega_m^t(d) \not\leq \omega_m^t(c)$ for some $(c,d)\in S_{m,m}$,   
and likewise with
$\omega^t_m$ replaced by $\omega^f_m$. 
Since $S_{m,m} = \Delta_m\cup \{ \, (a,b) \mid a \leq_k b \leq_k \top_{\!\!m}\,\}$
and $\omega_m^t$ and $\omega_m^f$ are order-preserving, 
these claims are easily verified.  

We now prove (ii)(b). 
First take  $j=m$ and $\omega\ne \omega' $ and 
assume that there 
existed $\mathbf{r} \in \mathbb S(\K_m^2)$ such that $r \in  R_{\omega,\omega' }$.  Necessarily 
$\mathbf{r} \supseteq \Delta_m$.  But  $\omega(p) \nleqslant \omega' (p)$
either for $p = \mbt_m$ or for $p = \mbf_m$. Hence $R_{\omega,\omega'}= \emptyset$. 
 Now consider $j > m$ and consider $\mathbf{r}$ in $\mathbb{S}(\K_j \times \K_m)$.  By Proposition~\ref{prop:Snn-subalg}(iv) 
there exists $\mathbf{s}\in \mathbb{S}(\K_j \times \K_j)$ such that
$r = \{ \, (a, h_{j,m}(b)) \mid   (a,b) \in s \,\}$.  
Hence  $(\mbt_j, \top_{\!\!m+1}),(\mbf_j,\top_{\!\!m+1}) \in r$. 
Then $\omega_j^t(\mbt_j)=1\not\le 0= \omega' ( \top_{\!\!m+1})$, which implies that $r \notin R_{\omega_j^t,\omega '}$. Similarly,  $\omega_j^f(\mbf_j)=1\not\le 0= \omega' ( \top_{\!\!m+1})$ implies that
$r \notin R_{\omega_j^f, \omega '}$, which concludes the proof of (b).

Finally we prove (ii)(c).   Assume $j < m$.  Then, for any  
$r \in R_{\omega,\omega' }$, we have $\mathbf{r} \in \mathbb{S}(\K_j \times 
\K_m)$ and hence there exists $\mathbf{s} \in \mathbb{S}(\K_m^2)$ 
such that 
$r = \{ \, (h_{m,j}(c),b) \mid  (c,b) \in s\,\}$.  
It is easily seen that $r \subseteq  \{\, (a,b) \mid \omega(a) \leq \omega ' (b)\,\}$ when 
$\mathbf{s} =\Delta_m^2$.  Hence 
$r$ 
is the converse of the graph of  $h_{m,j} $ and it
belongs to   $R_{\omega, \omega' }$. 
If 
$s\subseteq S_{m,j}\cap\conv{S}_{m,j}$, 
then 
$r=(\graph h_{m,j})^\smallsmile$. 
If 
$s\not\subseteq S_{m,j}\cap\conv{S}_{m,j}$, 
Theorem~\ref{thm:Snmonly} implies 
that $(\top_{\!\!j},\top_{\!\!m+1})\in
s$. 
Then $(\top_{\!\!j},\top_{\!\!m+1})\in r$. It follows that 
$\omega(\top_{\!\!j})=1
\nleqslant 0=\omega' (\top_{\!\!m+1})$ and 
$r\notin R_{\omega,\omega' }$.
\end{proof}

The definition of $\preccurlyeq$ in the following theorem may 
appear complicated, but the intuition behind it is quite simple.  
We consider the  natural dual space of an algebra $\A \in \CV_n$. 
This is a multisorted structure of the type described in 
Theorem~\ref{Thm:CharDualVn}. 
We first `double up'  each sort $X_m$ and give the doubled-up set 
an order determined by the partial order  
$\leq_m$
and the maps 
$\omega_m^t$ and $\omega_m^f$.  
Then we use the maps $g_i$ to arrange these sets  in layers, 
in order of increasing~$m$.

\begin{thm} \label{Thm:RevEngVarietyKn}
Let\,  $\A \in\CV_n$ and let\, 
$\fnt{D}(\A)$ be the structure
\,
$\X=(X_0\du \ldots \du X_n 
;\leq_0,\ldots,\leq_n,g_1,\ldots, g_{n},\Tp)$.

Let 
\, $Y = \textstyle{\bigcup} \,\{\, 
 X_m \times \{ \omega_m^t, \omega_m^f\} \mid 0 \leq m\leq n\,\}$.
Define a relation $\preccurlyeq $ on $Y$ by 
\begin{align*}
(x,\omega) \preccurlyeq (y,\omega' ) &\Longleftrightarrow  
x \leq_i y \mbox{ and } \omega=\omega' \in\{\omega_i^t,\omega_i^f\}; 
\text{ or }
\\& \hspace*{2.5cm} 
(x = (g_{i+1}\circ \cdots \circ g_j )(y), \mbox{ and }
\, \omega\in\{\omega_i^t,\omega_i^f\}, \omega' \in\{\omega_j^t,\omega_j^f\}, \mbox{ for some  }i<j).
\end{align*}
Then 
$\Y = (Y; \preccurlyeq, \Tp)$ is a Priestley space isomorphic to\, $\fnt{HU}(\A)$.
\end{thm}

\begin{proof}   
It is a  consequence of \cite[Theorems~2.1 and~2.3]{CPcop} 
and the  separation condition established in 
Lemma~\ref{lem:piggy-reconcile}(i) that $\preccurlyeq$ is a 
quasi-order on $Y$ for which the partially ordered 
space obtained by quotienting by $\preccurlyeq \cap \succcurlyeq$
 is isomorphic in $\CP$ to $\fnt{HU}(\A)$. 

We now claim that Lemma~\ref{lem:piggy-reconcile}(ii) 
implies   that $\preccurlyeq$ is a partial order rather than just a 
quasi-order.  
Each of the two copies of the sort $X_m$ carries the pointwise
lifting of the partial order $S_{m,m}$, and no pair of elements, one 
from each copy, is related by $\preccurlyeq$.  
We view 
$(X_m \times \{ \omega_m^t\}) \du (X_m\times\{ \omega_m^f\})$
as constituting the 
$m^{\text{th}}$ level of $\fnt{HU}(\A)$, for $m=0,\ldots,n$.
Lemma~\ref{lem:piggy-reconcile}(ii)(b) and (c) tell us that, 
with respect to $\preccurlyeq$, no point at level~$m$ is related to a point at a strictly lower level.  
\end{proof}

In preparation for analysing  the structure of the spaces 
$\fnt{HU}(\A)$ in particular cases
we  present some order-theoretic 
constructions involved in building such spaces. 
Consider first  posets $S$ and $T$ and a map $\phi\colon 
T \to S$. 
 Assume  that~$\phi$ is \defn{semi-constant},
 in the sense that it maps each order component 
of $T$ to a singleton  (see condition (iii) in 
Theorem~\ref{Thm:CharDualVn});
 any such map  is necessarily order-preserving.
The  \defn{restricted linear sum} 
$S \oplus_\phi T$
will be the poset obtained by equipping  the disjoint union 
$ S \du T$ with the relation  
$\leq_S \cup \leq_T \cup \, (\graph \phi)^\smallsmile$;
here $\leq_S$ and $\leq_T$ are the partial orders on $S$ and $T$, 
respectively.  

Take $T \overset{\phi}{\longrightarrow}  S$ as above.
We can then form a new poset, which we denote by  
$\overline{S} \oplus_{\overline{\phi}}\overline{T}$ 
refer to as the \defn{doubling} of 
$S \oplus_{\phi} T$.   The construction goes as follows    
 Take the disjoint union $\overline{S}$ of copies $S^1$ and $S^2$ 
of $S$ and the disjoint union $\overline{T}$ of copies $T^1$ 
and $T^2$ of $T$.
We let $\phi $ induce in the obvious way maps 
$\phi^{i,j}\colon T^i \to S^j$  (for $i,j \in \{1,2\}$) and form the restricted 
linear sums $T_j \overset{\phi_{i,j}}{\longrightarrow}  S_i$.    
Pasting the
order relations  together in the obvious way 
by taking their union
we obtain $T \overset{\phi}{\longrightarrow}  S$.

The two constructions above can  unambiguously be extended to 
the situation in which we start from any  finite sequence  
$P_0, \ldots ,P_n$ of  posets and semi-constant maps 
$\phi_i \colon P_{i} \to  P_{i-1}$.  
We can first form an iterated restricted linear sum 
$P_0 \oplus_{\phi_1}P_{1} \oplus_{\phi_{2}} \dots 
\oplus_{\phi_{n-1}}   P_{n-1} \oplus_{\phi_n}  P_n$.
Pictorially, that is,  in terms of a Hasse diagram, 
we view this  poset as having $(n+1)$ layers. 
The $m^{\text{th}}$-layer is $P_m$, and the ordering between
 the  layer $P_{m-1}$ and the layer $P_{m}$ above it is determined 
by $\phi_m$, for $m \geq 1$.  
Now we  can apply the doubling construction,  
extended in the obvious way, to obtain a new poset 
$Q_0 \oplus_{\psi_1} Q_{1} \oplus_{\psi_{2}} \dots 
\oplus_{\psi_{n-1}}   Q_{n-1} \oplus_{\psi_n}  Q_n$, 
where $Q_i = \overline{P_i}$ for $0 \leq i \leq n$ and 
$\psi_i = \overline{\phi_i}$ for $1 \leq i \leq n$.

We can extend these ideas in the obvious way  
to the setting of the category $\CP$, 
replacing posets by Priestley spaces,  and requiring the linking maps 
between them to be continuous as well as semi-constant.
Observe that  the Priestley space 
$\Y$ 
in Theorem~\ref{Thm:RevEngVarietyKn} is obtained  
from the sorts of $\D(\A)$ in just the way we have been 
describing above.

We turn now to examples.

\begin{ex}   \label{ex:Kn}
{\rm Take  $
\CV_n=\ISP(\K_0, \ldots ,\K_n)$  	 
and consider the algebra $\K_n$.
Up to isomorphism, (the underlying poset of) 
$\fnt{HU}(\K_n) = \CCD(\fnt{U}(\K_n),\two)$
is obtained as the doubling of the 
(restricted) linear sum 
$\{ \mbf_0\} \oplus \{\mbf_1\}  \oplus \dots \oplus \{\mbf_n\}$,
and which, suggestively, we label as in 
Fig.~\ref{fig:exKn}.  
This is exactly what we obtain from
Theorem~\ref{Thm:RevEngVarietyKn} if we identify 
$\mbt_i$ and $\mbf_i$ with the characteristic functions
of their up-sets with respect to $\leq_k$ on $K_n$.  
(Note that  $\CV_n(\K_n,\K_m) = \{ h_{n,m}\}$ for $0 \leq m \leq n$, so that
$\D(\K_n)$ consists of $n+1$ singletons.)  
Example~\ref{ex:Ln} provides a complementary discussion  of this 
example in terms of our 
product  representation.  
There we shall consider the representation of $\K_n$ and not 
just of $\fnt{U}(\K_n)$.}
\end{ex}

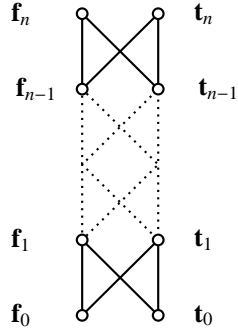
\begin{figure}[!t]
\begin{center}
\begin{tikzpicture}

\li{(5,0)--(5,1)--(6,0)--(6,1)--(5,0)}
\li{(5,3)--(5,4)--(6,3)--(6,4)--(5,3)}
\dotli{(5,1)--(5,2)--(6,1)--(6,2)--(5,1)}
\dotli{(5,2)--(5,3)--(6,2)--(6,3)--(5,2)}
\epo{5,0}
\epo{5,1}
\epo{5,3}
\epo{5,4}
\epo{6,0}
\epo{6,1}
\epo{6,3}
\epo{6,4}

\node at (4.2,0) {$\mbf_0$};
\node at (4.2,1) {$\mbf_1$};
\node at (4.4,3) {$\mbf_{n-1}$};
\node at (4.2,4) {$\mbf_n$};
\node at (6.6,0) {$\mbt_0$};
\node at (6.6,1) {$\mbt_1$};
\node at (6.8,3) {$\mbt_{n-1}$};
\node at (6.6,4) {$\mbt_n$};

\end{tikzpicture}
\end{center}
\caption{The Priestley dual of $\fnt{U}(\K_n)$ \label{fig:exKn}}
\end{figure}

\begin{ex}   
{\bf  (Priestley duals of reducts of free algebras in $\CV_n$)}  
We recall from Section~\ref{sec:multisorted}
 the fundamental fact that, in a natural duality for a class 
$\CA =\ISP(\CM)$ based on an alter ego $\CMT$, 
the free algebra $\Free_{\CA}(S)$ on a non-empty set $S$ of 
free generators is such that $\D(\Free_{\CA}(S)) = \CMT^s$  
(up to isomorphism in the topological quasivariety 
$\CX=\IScP(\CMT)$).

Let us apply Theorem~\ref{Thm:RevEngVarietyKn}  first to identify 
$\fnt{HU}(\Free_{\CV_n}(1))$ as a poset 
(its topology is discrete and plays no role). 
The required poset is obtained by applying doubling  to   
$K_0 \oplus_{h_{1,0}} K_{1} \oplus _{h_{2,1}} \dots 
\oplus_{h_{n-1,n-2}} K_{n-1} \oplus_{h_{n,n-1}}   K_n$. 
Here $K_m$ is equipped with the partial order $S_{m,m}$.  
With respect to this order it is
the disjoint union of an antichain with $3m$ elements and 
$\two^2$, where $\two$ denotes the two-element chain.   
Figure~\ref{fig:K1K0g1} 
shows the restricted linear sum 
$K_0 \oplus_{h_{1,0}} K_1$ 
and Fig.~\ref{fig:SEVEN-free-on-one} shows its doubling,
$\fnt{HU}(\Free_{\CV_1}(1))$.
In the figure, points shown by circles belong to level~$0$ and those by squares belong to level~$1$.

\begin{figure}[!	b]
\begin{center}
\begin{tikzpicture}
\path (-1,0) node(a) [rectangle,draw, minimum size=1.5pt, inner sep=3pt] {$\scriptstyle{\mbf_1}$}
			(0,1) node(b) [rectangle,draw, minimum size=1.5pt, inner sep=3pt] {$\scriptstyle{\top_{\!\!1}}$}
			(1,0) node(c) [rectangle,draw, minimum size=1.5pt, inner sep=3pt] {$\scriptstyle{\mbt_1}$}
			(0,-1) node(d) [rectangle,draw, minimum size=1.5pt, inner sep=3pt] {$\scriptstyle{\top_{\!\!2}}$}
			(3,0) node(e) [rectangle,draw, minimum size=1.5pt, inner sep=3pt] {$\scriptstyle{\mbf_0}$}
			(4,0) node(f) [rectangle,draw, minimum size=1.5pt, inner sep=3pt] {$\scriptstyle{\top_{\!\!0}}$}
			(5,0) node(g) [rectangle,draw, minimum size=1.5pt, inner sep=3pt] {$\scriptstyle{\mbt_0}$}
			(7,0) node(h) {$(K_1,S_{1,1})$}
		
			(4,-3) node(i) [circle,draw,inner sep=1.5pt] {$\scriptstyle{\top_{\!\!0}}$}
			(3,-4) node(j) [circle,draw,inner sep=1.5pt] {$\scriptstyle{\mbf_0}$}
			(5,-4) node(k) [circle,draw, inner sep=1.5pt] {$\scriptstyle{\mbt_0}$}
			(4,-5) node(l) [circle,draw,inner sep=1.5pt] {$\scriptstyle{\top_{\!\!1}}$}
			(7,-4) node(m) {$(K_0,S_{0,0})$}
			(6,-2.5) node(n) {$h_{1,0}$};

	\draw (a)--(b)--(c)--(d)--(a);
	\draw (i)--(j)--(l)--(k)--(i);
	\draw[loosely dashed] (-2,-2)--(6,-2);
	\path[-latex] (3,-0.4) edge node {} (3,-3.5);
	\path[-latex] (4,-0.4) edge node[above] {} (4,-2.5);
	\path[-latex] (5,-0.4) edge node[above] {} (5,-3.5);
	\path[-latex] (1,-0.4) edge[bend right=35] node[above] {} (3.5,-4.9); 
	\path[-latex] (0,-1.4) edge[bend right] node[above] {} (3.5,-5); 
	\path[-latex] (-1,-0.4) edge[bend right] node[above] {} (3.5,-5.1); 
	\path (-0.5,1) edge[bend right=40] node {} (-1.7,-0.5);
	\path (-1.7,-0.5) edge[bend right=25] node {} (0,-4);
	\path[-latex] (0,-4) edge[bend right=15] node {} (3.5,-5.2);
\end{tikzpicture}
\caption{The restricted linear sum 
$K_0 \oplus_{h_{1,0}} K_1$   
\label{fig:K1K0g1}}
\end{center}
\end{figure}
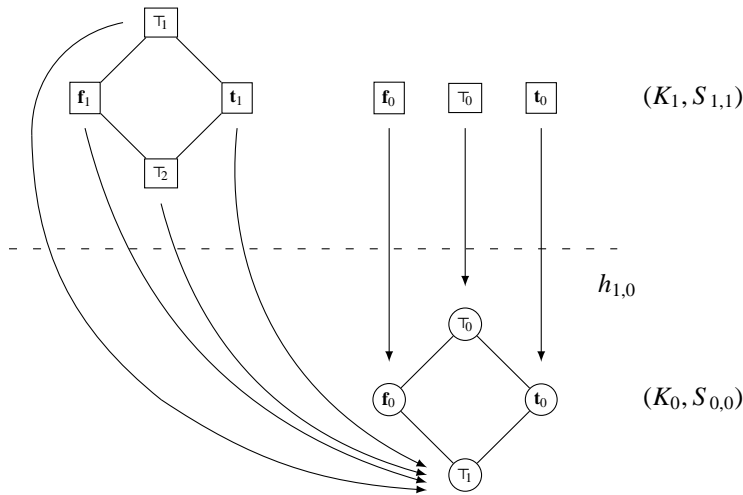

\begin{figure}[ht]
\begin{center}
\begin{tikzpicture}[scale=0.7]
\li{(-2,3)--(-3,4)--(-2,5)--(-1,4)--(-2,3)}
\li{(8,3)--(7,4)--(8,5)--(9,4)--(8,3)}
\li{(0,-1)--(-1,0)--(0,1)--(1,0)--(0,-1)}
\li{(5,0)--(6,1)--(7,0)--(6,-1)--(5,0)}

\draw (0,-1)--(-2,3)--(6,-1)--(8,3)--(0,-1);
\draw (-1,0)--(0,4)--(7,0);
\draw (0,1)--(0.75,4)--(6,1);
\draw (1,0)--(2,4)--(5,0);
\draw (5,0)--(4,4)--(1,0);
\draw (6,1)--(5.25,4)--(0,1);
\draw (7,0)--(6,4)--(-1,0);

\pdtwol{-3,4} 
\pdtwol{-2,5} 
\pdtwol{-1,4} 
\pdtwol{-2,3} 
\pdtwol{0,4} 
\pdtwol{0.75,4}
\pdtwol{2,4} 
\pdtwor{7,4} 
\pdtwor{8,5} 
\pdtwor{9,4} 
\pdtwor{8,3} 
\pdtwor{4,4} 
\pdtwor{5.25,4} 
\pdtwor{6,4} 

\pdonel{-1,0} 
\pdonel{0,1} 
\pdonel{1,0} 
\pdonel{0,-1} 
\pdoner{5,0} 
\pdoner{6,1} 
\pdoner{7,0} 
\pdoner{6,-1} 
\end{tikzpicture}
\end{center}
\caption{The Priestley dual space $\fnt{HU}(\Free_{\CV_1}(1))$}
\label{fig:SEVEN-free-on-one}
\end{figure}

Now let us describe  $\fnt{HU}(\Free_{\CV_n}(k))$, where  
$k$ is finite  (we consider the infinite case below).
The critical point is that $\D(\Free_{\CV_n}(k))$ may be identified 
with $\CMT_n^k$, with the power being calculated  `by  sorts'. 
 Once this is done, the translation to the Priestley dual 
$\fnt{HU}(\Free_{\CV_n}(k))$ proceeds as described in
Theorem~\ref{Thm:RevEngVarietyKn}.  
The $m^{\text{th}}$-layer  of $\fnt{HU}(\Free_{\CV_n}(k))$ is $K_m^k$. 
This can be obtained by induction on~$k$.
The linking map from  $K_m^k \to K_{m-1}^k$ is the $k$-fold
product map $h_{m,m-1} \times \dots \times h_{m,m-1}$.  
We can then describe $\fnt{HU}(\Free_{\CV_n}(k))$, using doubling, 
in the same way as for the case $k=1$.  
\end{ex} 

Our next task is to reveal how to recover $\fnt{U}(\A)$ from 
$\fnt{HU}(\A)$, for $\A \in \CV_n$, taking advantage of the layered 
structure of this Priestley dual space.  
For simplicity we shall first confine our remarks to the case $n=1$
and to finite~$\A$,  so that we are dealing with posets with two layers.  
We need to describe the  up-sets of a poset of the form 
$Q_0 \oplus_{\psi} Q_1$, obtained by doubling from a poset 
$P_0 \oplus_\phi P_1$, where $\phi\colon P_1 \to P_0$ is 
semi-constant.

First consider $\mathcal U(P)$, the family of up-sets of 
$P$, where $P=P_0 \oplus_\phi P_1$.  
Every set in $\mathcal U(P)$
takes the form
${(V_0 \cup \phi^{-1}(V_0)) \cup V_1}$,
where $V_i$ is an up-set in $P_i$ (for $i = 0,1$)  
and, since $\phi$ is semi-constant, we can choose $V_0$ and $V_1$ 
such that $ \phi^{-1}(V_0)\cap V_1 = \emptyset$;  
distinct pairs $(V_0,V_1)$ give rise to distinct up-sets of~$P$.  
Describing $\mathcal{U}(P)$ in full can be a complicated task. 
We note however  that we can easily get crude estimates
for the cardinality of $\mathcal{U}(P)$ when $P_0$
and $P_1$ are finite. We have 
\[
  |\mathcal{U}(P_0)| + |\mathcal{U}(P_1)| -1 \leq |\mathcal{U}(P)| \
\leq |\mathcal{U}(P_0) | \times |\mathcal{U}(P_1)|.
\]
The upper and lower bounds 
 come from consideration of, 
respectively,  the linear sum $P_0 \oplus P_1$ and the disjoint 
union $P_0 \du P_1$.  

\begin{ex}  \label{ex:free-on-1-V1}\quad 
{\bf  (The free algebras $\Free_{\CV_n}(1)$)} 
We first take  $n=1$. 
For $i=0,1$, we 
consider $K_i$ equipped with the partial order $S_{i,i}$. 
Each is the Priestley dual of 
a member $\Lalg_i$ of $\CCD$, for  $i=0,1$.  
By elementary Priestley duality, $\Lalg_0$ is  the linear sum  
$\boldsymbol 1 \oplus \two^2 \oplus \boldsymbol 1$, that is, 
a four-element Boolean lattice with new bottom and top elements 
adjoined.  
 The lattice $\Lalg_1$ equals  
 $\two^3 \times \Lalg_0$. 
 Then $\fnt{U}(\Free_{\CV_1}(1))$ is a $\CCD$-sublattice of $\Lalg_0^2
\times \Lalg_1^2$.

By calculating the number of up-sets of its Priestley dual, as shown in
Fig.~\ref{fig:SEVEN-free-on-one},
 we obtain
$|\Free_{\CV_1}(1)| = 5879$. 
We can compare this  value with our crude upper and lower bounds;
$
|\Lalg_0^2 |+ | \Lalg_1^2|-1 = 2339
$
and $|\Lalg_0^2 \times \Lalg_1^2 | = 82944$.
We also draw attention to the difference between the size of $\Free_{\CV_1}(1)$ and  
that of $\Free_{\CV_0}(1)$, which is~$36$ (the  Priestley dual is $\two^2 \du \two^2$).  
Two factors are at work here:  the passage to a strictly larger variety
and, perhaps more significantly, the weakening of the relation of
equivalence between bilattice terms  as a result of loss of distributivity.  

We can quickly see how 
$\fnt{HU}(\Free_{\CV_n}(1))$
is obtained order-theoretically from  
$\fnt{HU}(\Free_{\CV_{n-1}}(1))$, 
for $n \geq 1$, by adding a new top layer.  This gives 
$|\Free_{\CV_n}(1)|\geq|\Free_{\CV_{n-1}}(1)|+36(2^6)^{n}$. 
Hence, we can obtain a lower bound for $|\Free_{\CV_n}(1)|$ as follows
 \[
|\Free_{\CV_n}(1)|\geq 36 \frac{(2^6)^{n+1}-1}{2^6-1}\geq \frac12 (2^6)^{n+1}.
\]
\end{ex} 

So far we have looked at  $\fnt{HU}(\A)$ for  $\A $  a finite algebra in $\CV_n$ 
  and investigated in particular  the lattice $\fnt{U}( \Free_{\CV_n}(1))$.
We now make some comments applicable to arbitrary algebras.    
For infinite $\A$ the order components  within each layer are 
clopen in $\fnt{HU}(\A)$ and no significant issues arise in 
passage from the finite to the infinite case. 
Let $\A\in\CV_n$.  
 Associated with $\Y=\fnt{HU}(\A)$ is another 
Priestley space $\Z$ obtained by deleting the order relations 
between the layers.
It is 
a disjoint union of  Priestley spaces $\X_i$ ($0 \leq i \leq n$), 
where  each of these is the disjoint union of a Priestley space 
$\spc{W}_i$ with itself.   
  The distributive lattice $\fnt{K}(\Z)$  is therefore a product 
$\Lalg_0^2 \times \ldots \times \Lalg_n^2$ for certain 
$\Lalg_0, \ldots,\Lalg_n\in \CCD$.  
By basic Priestley duality, $\fnt{U}(\A)$ is a $\CCD$-sublattice 
of $\fnt{K}(\Z)$.  

Consider $\Free_{\CV_n}(\kappa)$, where now $\kappa$ is infinite.
The component layers of  
$\D(\Free_{\CV_n}(\kappa))$ (for $0 \leq m \leq n$) are the Priestley space powers 
$\spc{K}_m^\kappa$,
where $K_m$  carries the discrete topology; 
the linking map from $K_m^\kappa$ to  $K_{m-1}^\kappa$
 is determined by its compositions with the coordinate projections;
 each of these compositions is~$h_{m,m-1}$.

Theorem~\ref{Thm:RevEngVarietyKn}  gives  us access to a 
concrete representation of the $\CCD$-reduct $\fnt{U}(\A)$ for 
$\A \in\CV_n$, but not in a way which encodes the full bilattice 
structure.  
We remedy this omission in Section~\ref{sec:ProdRep}
by presenting our 
product representation.  
This will rely on showing how $\fnt{U}(\A)$ regarded as a  
sublattice of $\Lalg_0^2 \times \ldots \times \Lalg_n^2$ 
supports operations  $\wedge$,  $\vee$ and $\neg$
(the ones suppressed by $\fnt{U}$).  We thereby arrive at an algebra
isomorphic to the original bilattice~$\A$.

We were led to consider 
$\Lalg_0^2 \times \ldots \times \Lalg_n^2$ 
when we deleted the ordering between successive layers of the 
Priestley dual space $\fnt{HU}(\A)$.  
This ordering is derived from (the pointwise 
lifting to $\D(\A)$ of) the maps $g_i$, for $1 \leq i \leq n$ (see
Theorem~\ref{Thm:CharDualVn} 
as it applies to $\X = \D(\A)$).
The following elementary lemma reveals  the lattice-theoretic 
content of condition (iii)  in that theorem.

\begin{lem} \label{lem:Bool-sub} 
Let\,  $\Lalg, \M \in \CCD$ and let $\X= \fnt{H}(\Lalg)$ and 
$\Y=\fnt{H}(\M)$ be the Priestley dual spaces of\, $\Lalg$ and\, 
$\M$. 
 Let $f\in \CCD(\Lalg,\M)$ and let $\phi  = \fnt{H}(f)$ be the 
dual map.
Then the following statements are equivalent:
\begin{newlist}
\item[{\upshape (1)}]
$\phi \colon \Y \to \X$ is semi-constant;   
\item[{\upshape (2)}]
each element of  $f(\Lalg)$ has a complement in\, $\M$. 
\end{newlist}
\end{lem} 
\begin{proof}  
It will be convenient to  identify $\Lalg$ and $\M$ with 
the clopen up-sets of~$\X$ and~$\Y$, respectively, and to regard 
$f$ as being given by  $f(V) = \phi^{-1}(V)$, for each clopen up-set 
$V$ of $\Lalg$.  
Re-stated  in these terms,  (2) becomes the statement
that $\phi^{-1}(V)$ is a down-set for each clopen up-set~$V$ 
in~$Y$.
Assume that this is false for some $V$.
 Then we would be able to find 
$p$ and~$q$ in $Y$ with $p < q$, $p \notin \phi^{-1}(V)$ and
 $q \in \phi^{-1}(V)$.  
But this is incompatible with~(1). 

Conversely, assume~(1) fails. 
Then there exist~$p$ and~$q$ in~$Y$ with $p < q$ but 
$\phi(p) \ngeqslant \phi(q)$; here we have used the fact
that $\phi$ is order-preserving.  
Since~$\X$ is a Priestley space, there exists 
a clopen up-set~$W$ in~$\X$ with $\phi(q) \in W$ and 
$\phi(p) \notin W$. 
 Then $V = \phi^{-1}(W)$ is a clopen up-set in~$\Y$.  
But consideration 
of~$p$ and~$q$ shows that~$V$ cannot be a down-set.  
\end{proof} 
  
In the lemma the idea of levels of default seems very distant.  
Interestingly,  we shall see shortly that 
 condition~(2) emerges in a  natural way in the context of reasoning 
with defaults and helps to motivate the construction underlying 
our 
product representation.

\section{The Product Representation Theorem} 
\label{sec:ProdRep}
As we stressed at the outset,  our objective is to obtain a 
product-style representation for the members of $\CV_n$.
Our analysis of $\CCD$ reducts in the preceding section gives 
pointers as to how this might work, but did not give the full-blown 
representation we seek, because we did not encompass the 
operations suppressed by the forgetful functor~$\fnt{U}$.

Before introducing the formalism we shall employ,
we give some  intuition behind the construction of a bilattice 
from  a (not necessarily distributive) lattice 
(see \cite{BD13} for background on the construction).   
Consider a  situation in which the lattice $\Lalg$ arises as a  lattice 
of possible evidence (for example collected for a trial).  
Then, given a certain statement $s$, we assign to it a pair $(a,b)\in\Lalg^2$, where $a$ denotes the evidence in favour of $s$ 
(the \emph{positive evidence})
and $b$ the evidence against $s$ (the \emph{negative evidence}). 
Clearly we could have non-empty intersection 
between the positive and negative evidence for~$s$,
depending on the interpretation, and we may also have 
evidence that is neither for $s$ nor against it. 
The product $\Lalg\times \Lalg$ admits
two orderings.  Let $s = (a,b)$ and $s'=(b ,b')$ be elements of 
$\Lalg \times \Lalg$. In the knowledge order,  
$s \leq_k s'$ if $a\leq b $ and $b\leq b'$ in $\Lalg$, meaning that  
both the positive and negative evidence for $s$  is no greater than 
the corresponding evidence for~$s'$.
In the truth order,  $s\leq_ts'$  if $a\leq b $ and $b'\leq b$, 
meaning that the positive evidence for~$s$ is no greater than 
that for~$s'$, and the opposite holds for the negative evidence.   
There is also a natural interpretation of negation in this  set-up, 
given by $\neg(a,b)=(b,a)$, so that  whatever evidence is in favour 
of a statement is against its negation and vice versa.
Thus  we obtain a bilattice structure $\Lalg\odot \Lalg$ whose 
universe is $L\times L$.

It is natural to consider evidence being accumulated in an iterative 
fashion, with  pre-existing evidence taken into account by default, 
and  to be seen as  having priority. 
Such earlier evidence, encoded  by a lattice $\Lalg'$, might have 
come from  statistics, previous trials, or from other sources.
So new information, captured by a 
lattice $\Lalg$,  might be of a kind different from that  encoded by $\Lalg'$ 
but should be 
assumed to be  somehow connected to it.  
The connection between $\Lalg$ and $\Lalg'$ can be modelled by 
 a homomorphism from $\Lalg$ into $\Lalg'$, 
meaning that the evidence encoded in $a\in\Lalg$ has precedence 
over  any information that is contained in $h(a)$. 
For example, if $a$ and $b\in\Lalg$ represent the 
positive and negative evidence for a certain statement, then 
any default information (positive or negative) that is 
less informative than $h(a\vee b)$ does not add truly new 
information. 
Therefore, if the evidence about a statement is encoded by 
$(a,b)\in\Lalg\times\Lalg$ and the default information  that we get 
about the same statement is encoded by 
$(c,d)\in\Lalg'\times\Lalg'$,
then $h(a\vee b)\leq c$ and $h(a\vee b)\leq d$.
In this fashion the evidence {\em overrides} any positive or negative 
default information that we have about certain statements.

We shall now 
present our 
product representation.
We begin by setting up an equivalence between 
$\IScP(\twiddle{\CM_n})$, 
as described in Theorem~\ref{Thm:CharDualVn},  and another 
category related to $\CV_n$.  
This  equivalence implicitly subsumes parts of the `doubling' 
framework presented in Section~\ref{sec:varKn}
and provides a convenient formalism for developing our  theory 
in an algebraic setting.  

For $n\geq 0$, we define an \defn{$n$-default sequence} to be a 
sequence
\[
\Seq=\Lalg_0\overset{h_1}{\longrightarrow}\Lalg_1\overset{h_2}{\longrightarrow}\cdots\overset{h_{n-1}}{\longrightarrow}\Lalg_{n-1}\overset{h_n}{\longrightarrow}\Lalg_n,
\]
where 
$\Lalg_{i}\in \CCD$, for $0 \le i \le n$, 
and $h_{j} \in \CCD(\Lalg_{j-1}, \Lalg_j)$, for $1 \leq j\leq n$,  
are such that each $c$ in  $h_j(\Lalg_{j-1})$ has a  complement $c^*$  in~$\Lalg_j$.   
  The $n$-default sequences  support  a natural 
categorical structure.  
More precisely: an \defn{$n$-default morphism} $\boldsymbol{f}$
from $\Seq$ to $\Seq'$ is an $(n+1)$-tuple of $\CCD$-morphisms 
 $(f_0,\ldots,f_n)$ such that the diagram in 
Fig.~\ref{Fig:Sequences} commutes. 
It is easy to see that the class $\cat{DS}_n$ of $n$-default sequences with $n$-default morphisms is indeed a category.

 \begin{figure}  [ht]
\begin{center}
\begin{tikzpicture} 
[auto,
 text depth=0.25ex,
 move up/.style=   {transform canvas={yshift=1.9pt}},
 move down/.style= {transform canvas={yshift=-1.9pt}},
 move left/.style= {transform canvas={xshift=-2.5pt}},
 move right/.style={transform canvas={xshift=2.5pt}}] 
\matrix[row sep= 1cm, column sep= 1cm]
{ 
\node (S0) {$\Lalg_0$}; & \node (S1) {$\Lalg_1$}; &\node (M){$\ \cdots\ $}; &\node (Sn1) {$\Lalg_{n-1}$}; & \node (Sn) {$\Lalg_{n}$};\\ 
\node (T0) {$\Lalg'_0$}; & \node (T1) {$\Lalg'_1$}; &\node (N) {$\ \cdots\ $}; &\node (Tn1) {$\Lalg'_{n-1}$}; & \node (Tn) {$\Lalg'_{n}$};\\ 
};
\draw [->] (S0) to node {$h_0$}(S1);
\draw [->] (S1) to node {$h_1$}(M);
\draw [->] (M) to node {$h_{n-1}$}(Sn1);
\draw [->] (Sn1) to node {$h_n$}(Sn);
\draw [->] (T0) to node [swap] {$h'_0$}(T1);
\draw [->] (T1) to node [swap] {$h'_1$}(N);
\draw [->] (N) to node [swap] {$h'_{n-1}$}(Tn1);
\draw [->] (Tn1) to node [swap] {$h'_n$}(Tn);
\draw [->] (S0) to node [swap]{$f_0$}(T0);
\draw [->] (S1) to node [swap]{$f_1$}(T1);
\draw [->] (Sn1) to node [swap]{$f_{n-1}$}(Tn1);
\draw [->] (Sn) to node [swap]{$f_n$}(Tn);
\end{tikzpicture}
\end{center}
\caption{$n$-default morphism}\label{Fig:Sequences}
\end{figure} 

By Theorem~\ref{Thm:CharDualVn}, a simple extension of the functors $\fnt{H}$ and $\fnt{K}$ 
determines a dual equivalence  between $\IScP(\twiddle{\CM_n})$ 
and $\cat{DS}_n$.
This is set up by the functors $\fnt{H}_n\colon \cat{DS}_n\to \IScP(\twiddle{\CM_n})$ and $\fnt{K_n}
\colon  \IScP(\twiddle{\CM_n})\to\cat{DS}_n$.The functor $\fnt{H}_n$ is 
defined as follows:

\noindent 
on objects:
\quad 
 if $
\smash 
\Seq=\Lalg_0\overset{h_1}{\longrightarrow}\Lalg_1\overset{h_2}{\longrightarrow}\cdots\overset{h_{n-1}}{\longrightarrow}\Lalg_{n-1}\overset{h_n}{\longrightarrow}\Lalg_n,$ and $\fnt{H}(\Lalg_i)=(X_i;\leq_i,\Tp_i)$ for 
$i\in\{0,\ldots,n\}$ then
\[
\fnt{H}_n(\Seq)=(X_0\du \cdots \du X_n;\leq_0,\ldots,\leq_n,\fnt{H}(h_1),\ldots,\fnt{H}(h_n),\Tp),
\] 
where $\Tp$ is the disjoint union topology; 

\noindent 
on morphisms:  \quad  
$\fnt{H}_n(f_0,\ldots,f_n)=\fnt{H}(f_0)\du\cdots\du\fnt{H}(f_n)$.

\noindent In the other direction,  we define $\fnt{K}_n$ as follows:  

\noindent 
on objects: \quad 
if $\X=(X_0\du \cdots \du X_n;\leq_0,\ldots,\leq_n,g_1,\ldots,g_n,\Tp)$ and $\Lalg_i=\fnt{K}(X_i;\leq_i,\Tp_i)$ for 
$i\in\{0,\ldots,n\}$ then
\[\fnt{K}_n(\X)=\Lalg_0\overset{\fnt{K}(g_1)}{\longrightarrow}
\Lalg_1  \overset{\fnt{K}(g_2)}{\longrightarrow} \cdots  \overset{\fnt{K}(g_{n-1})}{\longrightarrow} \Lalg_{n-1}
\overset{\fnt{K}(g_{n})}{\longrightarrow}\Lalg_n;
\]
on morphisms:   \quad 
$ \fnt{K}_n(u)= (\fnt{K}(u\restrict_{X_0}),\ldots,\fnt{K}(u\restrict_{X_n}))$, where $u\restrict_{X_i}\colon X_i\to \K_i$ is the restriction of $u$ to 
$X_i$ for $i\in\{0,\ldots,n\}$.

\begin{figure}  [ht]
\begin{center}
\begin{tikzpicture} 
[auto,
 text depth=0.25ex,
 move up/.style=   {transform canvas={yshift=1.9pt}},
 move down/.style= {transform canvas={yshift=-1.9pt}},
 move left/.style= {transform canvas={xshift=-2.5pt}},
 move right/.style={transform canvas={xshift=2.5pt}}] 
\matrix[row sep= 1cm, column sep= 1.7cm]
{ 
\node (DB) {$
\CV_n = \ISP(\CM_n) $};
; & \node (X) {$\IScP(\twiddle{\CM_n})$}; & \node (DS) {$\cat{DS}_n$};\\ 
};
\draw  [->, move up] (DB) to node  [yshift=-2pt] {$\D$}(X);
\draw [<-,move down] (DB) to node [swap]  {$\E$}(X);
\draw  [->,move up] (X) to node  [yshift=-2pt] {$\fnt{K}_n$}(DS);
\draw [<-,move down] (X) to node [swap]  {$\fnt{H}_n$}(DS);
\end{tikzpicture}
\end{center}
\vspace*{-.2cm}
\caption{Equivalence between $\CV_n$ and $\cat{DS}_n$}\label{Equivalences}
\end{figure}

The discussion above 
paves the way to 
our construction of a 
\defn{product default bilattice}. 
Given an $n$-default sequence 
$\Seq=
\Lalg_0\overset{h_1}{\longrightarrow}\cdots\overset{h_n}{\longrightarrow}\Lalg_n$,
we shall define a default bilattice 
$\Seq\odot\Seq$ isomorphic in $\CV_n$ to 
$\fnt{E}\circ\fnt{H}_n(\Seq)$  whose universe $A$ is   
included in $\Lalg_0^2\times\cdots\times\Lalg_n^2$.
If $a\in\Lalg_0^2\times\cdots\times\Lalg_n^2$, then 
$(a_{i,\mbt},a_{i,\mbf})\in \Lalg_i^2$ denotes the pair formed 
from  the $(2i+1)$- and $(2i+2)$-coordinates of $a$.  
We define
\allowdisplaybreaks
\[\
A=\{\,a\in\Lalg_0^2\times\cdots\times\Lalg_n^2\mid h_i(a_{i-1,\mbt}\vee a_{i-1,\mbf})\leq a_{i,\mbt},a_{i,\mbf}\mbox{ for }1\leq i\leq n\,\}.
\]
We  
define $\iota\colon A\to \fnt{E}\circ \fnt{H}_{n}(\Seq)$ recursively 
in the following way.  Fix $a \in A$ and let $z\in\fnt{H}(\Lalg_{i}) $.

\noindent {\bf Case 1}:  
Assume   that  either $i=0$ or that $i > 0$ and  $\iota(a)(z\circ h_i)=\top_{\!\!i}$.  Then we define
\[
\iota(a)(z)=\begin{cases}
\top_{\!\!i}&\mbox{ if }z(a_{i,\mbt})=z(a_{i,\mbf})=1,\\
\mbt_i&\mbox{ if }z(a_{i,\mbt})=1\mbox{ and }z(a_{i,\mbf})=0,\\
\mbf_i&\mbox{ if }z(a_{i,\mbt})=0\mbox{ and }z(a_{i,\mbf})=1,\\
\top_{\!\!i+1}&\mbox{ if }z(a_{i,\mbt})=z(a_{i,\mbf})=0.
\end{cases}
\]

\noindent {\bf Case 2:}  
Assume that $i > 0$ and $\iota(a)(z\circ h_i)\neq \top_{\!\!i}$.  Then we define $\iota(a)(z)=\iota(a)(z\circ h_i)$.

\noindent Since each 
$x\in \fnt{H}(\Lalg_i)$ 
is continuous and order-preserving, 
so is  $\iota(a)$.
It is routine  to check that 
$\iota(a)\circ \fnt{H}(h_i) =h_{i,i-1}\circ\iota(a)$. 
Hence $\iota$ is well defined.  
Furthermore,  $\iota$ is bijective and its inverse is given 
as follows.
Let $f=f_0\du\cdots\du f_n$ belong to 
$\fnt{E}\circ \fnt{H}_n(\Seq)$ and $a_{i,\mbt}$ and $a_{i,\mbf}$ be the unique elements of 
$\Lalg_i$ determined by  the clopen up-sets 
 $f_i^{-1}({\uparrow}_{k}\mbt_i)$ and 
$f_i^{-1}({\uparrow}_{k}\mbf_i)$, respectively.  Then 
$ 
\iota^{-1}(f)=
((a_{0,\mbt},a_{0,\mbf}),\ldots,(a_{n,\mbt},a_{n,\mbf}))$.

We shall now use $\iota$
to define the bilattice 
operations of 
\[
\Seq\odot\Seq=  
\bigl( A; \otimes, \oplus, \land, \lor, \neg, \bot,\top\bigr)
\]
in such a way that 
$\Seq\odot\Seq\cong\fnt{E}\circ\fnt{H}_n(\Seq)$.
This is  done  simply by arranging that,  for each $a,b\in A$,
\[ 
a\star b
=\iota^{-1}(\iota(a)\star\iota(b))\  
\mbox{ for }\star\in\{\vee,\wedge,\oplus,\otimes\}; \quad \neg a = \iota^{-1}(\neg(\iota(a))); \quad  
\top=\iota^{-1}(\top) \mbox{ and }
\bot=\iota^{-1}(\bot).
\]
In what follows we present 
an alternative 
description
of the operations in $\Seq\odot\Seq$ in terms of the coordinates of the elements of $A$ involved and the homomorphisms $h_i$  of
 the sequence $\Seq$. This description is
intrinsic to $\Seq$ and 
does not refer to 
the map $\iota$.  
First we let 
\begin{alignat*}{2}
(a\otimes b )_{0,\mbt}&=
a_{0,\mbt}\wedge b _{0,\mbt},&
(a\otimes b )_{0,\mbf}&=
a_{0,\mbf}\wedge b _{0,\mbf}, \\
(a\oplus b )_{0,\mbt}&=
a_{0,\mbt}\vee b _{0,\mbt},\qquad \quad &
(a\oplus b )_{0,\mbf}&=
a_{0,\mbf}\vee b _{0,\mbf}, \\
(a\wedge b )_{0,\mbt}&=a_{0,\mbt}\wedge b _{0,\mbt},&
(a\wedge b )_{0,\mbf}&=
a_{0,\mbf}
\vee b _{0,\mbf},\\
(a\vee b )_{0,\mbt}&=
a_{0,\mbt}\vee 
b _{0,\mbt},&
(a\vee b )_{0,\mbf}&=
a_{0,\mbf}\wedge b _{0,\mbf}.
\end{alignat*}
Inductively, if we have  
specified  
$(a\star b )_{ i-1}$, 
then in  
the $i^{\text{th}}$
pair of coordinates, 
we take account of the previously encoded information
to obtain  $(a\star b )_i$.
We recall that  the definition of a default sequence assumes that, for 
$i > 0$,   the homomorphism 
$h_i\colon \Lalg_{i-1} \to \Lalg_i$ 
is such that each element~$c$ of 
$h_i(\Lalg_{i-1})$ has a complement~$c^*$ in $\Lalg_i$.
We let 
\begin{align*}
(a\otimes b )_{i,\mbt}&=
a_{i,\mbt}\wedge b _{i,\mbt},\\
(a\otimes b )_{i,\mbf}&=
a_{i,\mbf}\wedge b _{i,\mbf},
\\
(a\oplus b )_{i,\mbt}&=
a_{i,\mbt}\vee b _{i,\mbt},\\
(a\oplus b )_{i,\mbf}&=
a_{i,\mbf}\vee b _{i,\mbf},
\\[.8ex]
(a\wedge b )_{i,\mbt}&=
(a_{i,\mbt}\wedge b _{i,\mbt})\vee h_i((a\wedge b )_{i-1,\mbt}\vee (a\wedge b )_{i-1,\mbf}),
\\
(a\wedge b )_{i,\mbf}&=
(a_{i,\mbf}\wedge h_i(a_{i-1,\mbt})^*)\vee (b _{i,\mbf}\wedge h_i(b _{i-1,\mbt})^*)
\vee h_i((a\wedge b )_{i-1,\mbt}\vee (a\wedge b )_{i-1,\mbf}),
\\
(a\vee b )_{i,\mbt}&=
(a_{i,\mbt}\wedge h_i(a_{i-1,\mbf})^*)\vee (b _{i,\mbt}\wedge h_i(b _{i-1,\mbf})^*)  
\vee h_i((a\vee b )_{i-1,\mbt}\vee (a\vee b )_{i-1,\mbf}),
\\
(a\vee b )_{i,\mbf}&=
(a_{i,\mbf}\wedge b _{i,\mbf})\vee h_i((a\vee b )_{i-1,\mbt}\vee (a\vee b )_{i-1,\mbf}).
\end{align*}
In terms of coordinates, the negation operation is given 
 by
\[
(\neg a)_{i,\mbt}=a_{i,\mbf},\quad
(\neg a)_{i,\mbf}=a_{i,\mbt}  \qquad  \text{for  
} i\in\{0,\ldots,n\};
\]
and the constants  by
$\top = (1,\ldots,1)$ and $\bot=(0,\ldots,0)$.

We demonstrate for $\oplus$ and $\vee$ 
 that our alternative specifications fit with the  definitions  in terms of $\iota$
that we gave initially.
The remaining operations are handled similarly.
Given  $a,b \in A$ and $x_0\in\fnt{H}(\Lalg_{0})$,  
\begin{align*}
x_0(\iota^{-1}(\iota(a)\oplus \iota(b ))_{0,\mbt})=1
&\Longleftrightarrow    (\iota(a)\oplus \iota(b ))(x_0)\in\{\mbt_0,\top_{\!\!0}\}\\
&\Longleftrightarrow    \iota(a)(x_0)\in\{\mbt_0,\top_{\!\!0}\}\mbox{ or }\iota(b )(x_0)\in\{\mbt_0,\top_{\!\!0}\}\\
&\Longleftrightarrow    x_0(a_{0,\mbt})=1\mbox{ or }x_0(b _{0,\mbt})=1\\
&\Longleftrightarrow   x_0(a_{0,\mbt}\vee b _{0,\mbt})=1;\\
\intertext{likewise one can show }
x_0(\iota^{-1}(\iota(a)\oplus \iota(b ))_{0,\mbf})=1&\Longleftrightarrow    (\iota(a)\oplus \iota(b ))(x_0)\in\{\mbf_0,\top_{\!\!0}\}
\Longleftrightarrow   x_0(a_{0,\mbf}\vee b _{0,\mbf})=1; \\ 
x_0(\iota^{-1}(\iota(a)\vee \iota(b ))_{0,\mbt})=1
&\Longleftrightarrow    (\iota(a)\vee \iota(b ))(x_0)\in\{\mbt_0,\top_{\!\!0}\} 
\Longleftrightarrow   x_0(a_{0,\mbt}\vee b _{0,\mbt})=1;
\\ 
x_0(\iota^{-1}(\iota(a)\vee \iota(b ))_{0,\mbf})=1
&\Longleftrightarrow    (\iota(a)\vee \iota(b ))(x_0)\in\{\mbf_0,\top_{\!\!0}\}
\Longleftrightarrow   x_0(a_{0,\mbf}\wedge b _{0,\mbf})=1.
\end{align*}
Now consider $i \geq 1$. 
We give the first calculation
in full detail so as to bring out clearly how the 
definition of the universe~$A$ comes into play.  The second 
and subsequent calculations are more abbreviated but involve the same ideas.
Let $x_i\in \fnt{H}(\Lalg_{i})$.  Then  
\begin{align*}
x_i(\iota^{-1}(\iota(a)\oplus \iota(b ))_{i,\mbt})=1
&\Longleftrightarrow   (\iota(a)\oplus \iota(b ))(x_i)\in{\uparrow}_{k}{\mbt_i} 
\\
 &\Longleftrightarrow   \iota(a)(x_i)\in{\uparrow}_{k}{\mbt_i}\mbox{ or }  \iota(b )(x_i)\in{\uparrow}_{k}{\mbt_i}\\
 &\Longleftrightarrow   x_i(a_{i,\mbt})=1\mbox{  or }x_i(b _{i,\mbt})=1
\mbox{ or } \iota(a)(x_i\circ h_i)\neq \top_{\!\!i}\mbox{ or }\iota(b )(x_i\circ h_i)\neq \top_{\!\!i}\\
&\Longleftrightarrow   x_i(a_{i,\mbt})=1\mbox{  or }x_i(b _{i,\mbt})=1
 \mbox{ or } (x_i\circ h_i)(a_{i-1,\mbt})=1\mbox{ or }
(x_i\circ h_i)(a_{i-1,\mbt})=1
\\ &\hspace*{5.5cm} 
\mbox{or } (x_i\circ h_i)(b _{i-1,\mbt})=1\mbox{ or }
(x_i\circ h_i)(b _{i-1,\mbt})=1\\
&\Longleftrightarrow   x_i(a_{i,\mbt})=1\mbox{  or }x_i(b _{i,\mbt})=1 
\mbox{ or }
 x_i( h_i(a_{i-1,\mbt}\vee a_{i-1,\mbt}))=1
\mbox{  or } x_i( h_i(b _{i-1,\mbt}\vee b _{i-1,\mbt}))=1\\
&\Longleftrightarrow   x_i(a_{i,\mbt}\vee b _{i,\mbt})=1;\\[1.5ex]
x_i(\iota^{-1}(\iota(a)\oplus \iota(b ))_{i,\mbf})=1 
&\Longleftrightarrow   (\iota(a)\oplus \iota(b ))(x_i)\in{\uparrow}_{k}{\mbf_i}\\
&\Longleftrightarrow  
 \iota(a)(x_i)\in{\uparrow}_{k}{\mbf_i}\mbox{ or }  \iota(b )(x_i)\in{\uparrow}_{k}{\mbf_i}\\
&\Longleftrightarrow  
 x_i(a_{i,\mbf})=1\mbox{  or }x_i(b _{i,\mbf})=1
\mbox{ or } \iota(a)(x_i\circ h_i)\neq \top_{\!\!i}
 \mbox{ or }\iota(b )(x_i\circ h_i)\neq \top_{\!\!i}\\
&\Longleftrightarrow   x_i(a_{i,\mbf}\vee b _{i,\mbf})=1. 
\end{align*}
We now consider $\vee$.  
In the last step of the first calculation
we use the fact that $x_i\colon\Lalg_i\to\two$ is a lattice homomorphism
In the second calculation we need  
additionally 
the fact that, for a complemented element $c$ in $\Lalg_i$, we have 
$x_i(c) = 0 $ if and only if $x_i(c^*) = 1$.  We have 
\begin{align*}
x_i( 
\iota^{-1}(\iota(a)\vee \iota(b ))_{i,\mbt})=1 
& \Longleftrightarrow 
(\iota(a)\vee \iota(b ))(x_i)\in{\uparrow}_{k}{\mbt_i}\\
& \Longleftrightarrow   \iota(a)(x_i)\in
\{ \mbt_j\}_{j \leq i} \cup \, \{ \top_{\!\!j}\}_{j \leq i} 
\mbox{ or }
\iota(b )(x_i)\in
\{\mbt_j\}_{j \leq i} 
\cup \, \{\top_{\!\!j}\}_{j\leq i} 
 \mbox{ or }
\iota(a)(x_i),\iota(b )(x_i)\in
\{\mbf_j\}_{j\leq i-1} \\
&\Longleftrightarrow   \iota(a)(x_i)\in\{\mbt_j\}_{j \leq i-1} \cup\,
\{\top_{\!\!j}\}_{j\leq i-1} 
\mbox{ or }
\iota(b )(x_i)\in
\{\mbt_j \}_{j\leq i-1}  \cup\, \{ \top_{\!\!j}\}_{j \leq i-1} 
\\
&\hspace*{2.4cm} 
\mbox{ or } 
\{\iota(a)(x_i),\iota(b )(x_i)\} \subseteq 
\{ \mbf_j\} _{j\leq i-1} 
\iota(a)(x_i)\in\{\mbt_i,\top_{\!\!i}\}\mbox{ or }\iota(b )(x_i)\in
\{\mbt_i,\top_{\!\!i}\}
\\
 &\Longleftrightarrow   
   \iota(a)(x_i\circ h_i)\in
 \{\mbt_j\}_{j \leq i-1} \cup\, \{\top_{\!\!j}\}_{j \leq i-1}
 \mbox{ or }
  \iota(b )(x_i\circ h_i)\in
  \{\mbt_j\}_{j \leq i-1} \cup \, \{ \top_{\!\!j}\}_{j \leq i -1}
  \\
  & \hspace*{2.4cm} 
\mbox{ or }
 \iota(a)(x_i\circ h_i),\iota(a)(x_i\circ h_i)\in
  \{ \mbf_j\}_{j \leq i-1}
\mbox{ or }
\iota(a)(x_i)\in\{\mbt_i,\top_{\!\!i}\}\mbox{ or }
\iota(b )(x_i)\in\{\mbt_i,\top_{\!\!i}\} \\ 
  & \Longleftrightarrow    x_i(h_i((a\vee b )_{i-1,\mbt}))=1\mbox{ or }x_i(h_i((a\vee b )_{i-1,\mbf}))=1 \\
&  \hspace*{2.4cm}
\mbox{ or } 
  \bigl(x_i(a_{i,\mbt})=1\mbox{ and }x_i(h_i(a_{i-1,\mbf}))=0\bigr)
\mbox{ or }
\bigl(x_i(b _{i,\mbt})=1\mbox{ and }x_i(h_i(b _{i-1,\mbf}))=0\bigr)\\
&\Longleftrightarrow  x_i(a_{i,\mbt}\wedge h_i(a_{i-1,\mbf})^*)=1\mbox{ or } 
x_{i}(b _{i,\mbt}\wedge h_i(b _{i-1,\mbf})^*)=1 
\mbox{ or }
 x_i(h_i((a\vee b )_{i-1,\mbt}\vee (a\vee b )_{i-1,\mbf}))=1;
\\[1.5ex]
x_i(\iota^{-1}(\iota(a)\vee \iota(b ))_{i,\mbf})=1
&\Longleftrightarrow   (\iota(a)\vee \iota(b ))(x_i)\in{\uparrow}_{k}{\mbf_i}\\
&\Longleftrightarrow   \iota(a)(x_i)\in
\{\mbt_j\}_{j \leq i-1} \cup\, \{\top_{\!\!j}\}_{j \leq i-1} \mbox{ or }
\iota(b )(x_i)\in
\{\mbt_j\}_{j \leq i-1} \cup\, \{\top_{\!\!j}\}_{j\leq i-1}
\mbox{ or }
\{ \iota(a)(x_i), \iota(b )(x_i)\} \subseteq {\uparrow}_k \mbf_i 
\\
 &\Longleftrightarrow   \iota(a)(x_i\circ h_i)\in
\{ \mbt_j\}_{j\leq i-1} \cup\, \{\top_{\!\!j}\}_{j\leq i-1}
\mbox{ or } 
\iota(b )(x_i\circ h_i)\in\{
\mbt_j\}_{j\leq i-1}\cup\, \{ \top_{\!\!j}\}_{j \leq i-1} 
\\
&\hspace*{2.4cm} 
\mbox{ or }
\bigl( x_i(a_{i,\mbf})=1\mbox{ and }x_i( b _{i,\mbf})=1\bigr)  \\
&\!\! \Longleftrightarrow 
   x_i(h_i((a\vee b )_{i-1,\mbt}))=1\mbox{ or }
x_i(h_i((a\vee b )_{i-1,\mbf}))=1 \mbox{ or }
x_i(a_{i,\mbf}\wedge b _{i,\mbf})=1.
\end{align*}
Very little work is needed to complete the proof of 
the following theorem.

\begin{thm}\label{Thm:ProdRepDefault} 
{\bf (Product Representation Theorem)} 
For each default bilattice\,
$\B \in 
\CV_n$
there exists an $n$-default sequence\, $\Seq$ such that\, 
$\B\cong \Seq\odot\Seq$.
\end{thm}
\begin{proof}
It suffices to consider $\Seq=(\fnt{K}_n\circ\fnt{D})(\B)$ and  to
observe that, by definition of $\Seq\odot\Seq$ and the equivalence 
between $\HSP(\K_n)$ and $\cat{DS}_n$,  
\[
\B\cong (\fnt{E}\circ\fnt{H}_n)((\fnt{K}_n\circ\fnt{D})(\B))=
(\fnt{E}\circ\fnt{H}_n)(\Seq)\cong \Seq\odot\Seq.\qedhere
\]
\end{proof}

The  operations in a lattice determine, and are determined by, 
the underlying order.  
This leads us to enquire how the knowledge and truth orders of
$\Seq \odot \Seq$ can be  characterised. 
Our definitions of $\otimes$ and of $\oplus$  imply immediately that these 
operations on the subset~$A$ coincide with the coordinatewise-defined lattice operations 
on $(\Lalg_0^2 \times \dots \times \Lalg_n^2)$.  As a consequence,
the knowledge order $\leq_k$ of the bilattice~$\Seq \odot \Seq$
is the inherited coordinatewise order. (An application of Theorem~\ref{Thm:RevEngVarietyKn}, provides an alternative proof of this statement.)  
One may then ask whether the truth order $\leq_t$ can likewise be described in terms of coordinates.  Certainly, for $a,b \in \A$,  we have 
$a \leq_t b$ if and only if  $ a \lor b = b$
(or equivalently if and only if $ a \land b = a$).
The fact that $a \lor b = \iota^{-1 }(\iota(a) \lor \iota(b))$ implies that
$a \leq_t b$ if and only if $\iota(a) \lor \iota(b) = \iota(b)$. 
However an intrinsic description of $\leq_t$ is not easy to formulate in general.  

We  end this section with a simple  example, spelled out in detail.
We include  a recursive description of the truth order, capitalising on the 
fact that each lattice in our default sequence  has only two elements.

\begin{ex}  \label{ex:Ln}
Let $n\geq 1$ and $\Seq_n$ be the $n$-default sequence:
 \[
\Seq_n=
\underbrace{\two\overset{\Id}{\longrightarrow}\two\overset{\Id}{\longrightarrow}\ \cdots \
\overset{\Id}{\longrightarrow} \two\overset{\Id}{\longrightarrow}\two}_{n -{ \rm times}},
\]
where $h_i = \Id \colon\two\to\two$, the identity map, for each~$i$.
Here 
\begin{align*} 
A_n &=\{\,a\in \
2^{2(n+1)}\mid  h_i(a_{i-1,\mbt} \lor a_{i-1,\mbf}) \leq a_{i,\mbt}, a_{i,\mbf} \text{ for } 1 \leq i \leq n\,\}  \\
&= \{\,a\in \ 2^{2(n+1)}\mid \max 
\{a_{i-1,\mbt} 
,a_{i-1,\mbf}\} \leq \min\{a_{i,\mbt}, a_{i,\mbf}\} \text{ for } 1 \leq i \leq n\,\}.  
\end{align*}
No restrictions are imposed on $a_{0,\mbt}$ or $a_{0,\mbf}$, so there are 
four choices for this pair.  
However, if $1 \in \{a_{i,\mbt},a_{i,\mbf}\}$ then $a_{j,\mbt}=1=a_{j,\mbf}$ for all $j>i$. 
It follows that $A_n$ has  $4 + 3n$ elements.   

This calculation strongly suggests that $\Seq_n \odot \Seq_n$
is isomorphic to $\K_n$ (note  also  Example~\ref{ex:Kn}).
We would like to  show that  the map $\Phi_n$
defined by
\[
\Phi_n(a)=\begin{cases}
\top_{\!\!i}&\mbox{if }1=a_{i,\mbt}=a_{i,\mbf},\mbox{ and }a_{j,\mbt}=a_{j,\mbf}=0\mbox{ for each } j<i,\\
\mbt_i&\mbox{if }1=a_{i,\mbt} \mbox{ and }0=a_{i,\mbf},\\
\mbf_i&\mbox{if }1=a_{i,\mbf} \mbox{ and }0=a_{i,\mbt},\\
\top_{\!\!n+1}&
\mbox{if }
a_{j,\mbt}=a_{j,\mbf}=0\mbox{ for  }0\leq j\leq n
\end{cases}
\]
is an isomorphism of default bilattices.  
Certainly $\Phi_n$ is surjective.  Since its domain and range have the same cardinality, it is also injective.
We want  to show that $\Phi_n$ preserves all the bilattice operations.  
This is trivial for $\bot$, $\top$ and $\neg$ and routine for 
$\otimes$ and $\oplus$.  
We now apply our general formulae  to calculate 
$(a \wedge b )_{i,v}$ for $a,b \in A_n$, where  
$v \in \{ \mbt,\mbf\}$.  Since $(a \wedge b )_{i,v}\in \{0,1\}$, 
it will suffice to give the conditions under which the value is~$1$.
Note first that $(a \land b )_{0,\mbt} = 1$ if and only if
$a_{0,\mbt} =b _{0,\mbt} =1$ and that 
$(a \land b )_{0,\mbf} = 1$ if and only if 
$a_{0,\mbf} =1$ or $b _{0,\mbf} =1$.   
For $i > 1$,
\begin{align*} 
(a \land b)_{i,\mbt} = 1
&\Longleftrightarrow 
  (a_{i,\mbt} = 1 \text{ and  }  b _{i,\mbt}=1)
\,\text{ or }\,  \bigl((a\land b )_{i-1,\mbt} = 1  \text{ or } (a\land b )_{i-1,\mbf} = 1\bigr); \\
(a \land b)_{i,\mbf} = 1
&\Longleftrightarrow 
\bigl( a_{i,\mbf} = 1 \text{ and } a_{i-1, \mbt} = 0\bigr)
\,\text{ or }\,
\bigl( b _{i,\mbf} = 1 \text{ and } b _{i-1, \mbt} = 0\bigr)
\,\text{ or }\, \bigl((a\land b )_{i-1,\mbt} = 1 \text{ or } (a\land b )_{i-1,\mbf} = 1\bigr).
\end{align*}
Note that the recursive components of these definitions reflect 
the restriction imposed by  $a \land b $ belonging  to 
$\Seq _n\odot\Seq _n$.
Similar considerations apply to~$\lor$.  
It can now be verified that $\Phi_n$ preserves $\land$ and $\lor$.

We already know that $\Seq _n \odot \Seq _n$ inherits its knowledge order coordinatewise from  
$\two^{2(n+1)}$.  
To access the truth order,
we use the fact that 
 $a\leq_t b $   if and only if
$a_{i,\mbt} = (a\land b )_{i,\mbt}$ 
and 
$b _{i,\mbf} = (a\lor b )_{i,\mbf}$ 
for ${0\leq i\leq n}$. 
By treating odd and even coordinates in this way we are able to save work by exploiting the parallels between the definitions of $\land$ and $\lor$.
Certainly 
${a_{0,\mbt} = (a\land b )_{0,\mbt}} $ if and only if 
$a_{0,\mbt} \leq b _{0,\mbt}$ and 
$a_{0,\mbf} = (a\land b )_{0,\mbf} $ if and only if 
${a_{0,\mbf} \geq b _{0,\mbf}}$.  
Hence in identifying when $a \leq_t b $ it will be sufficient to restrict attention
to pairs $a$ and $b $ whose ${0,v}$-coordinates are related in this way.   
Consider $i=1$.  Taking into account the fact that $a$ and $b$ are members 
of~$A$, we  see easily that  
 \begin{align*} 
a_{1,\mbt} =  (a\land b )_{1,\mbt} 
&\Longleftrightarrow 
a_{0,\mbt} = 1
 \text{ or }  
a _{0,\mbf} = 1 \text{ or }  a_{1,\mbt} \leq b_{1,\mbt}
\intertext{and, likewise,} 
b_{1,\mbf} =  (a\lor b )_{1,\mbf} 
&\Longleftrightarrow 
b_{0,\mbt} = 1
 \text{ or } 
 b _{0,\mbf} = 1 \text{ or }  a_{1,\mbf} \geq b_{1,\mbf}.
\end{align*}
We now proceed by recursion, in the following way.
We let 
 $a_{\leq i}\in  \{0,1\}^{2(n+1)} $ be the projection of $a$ onto 
its 
first $2(i+1)$ coordinates. 
We have  
$a\leq_t b$ if and only if $ (a\wedge b)_{i,\mbt}=a_{i,\mbt}$ and 
$(a\lor b)_{i,\mbf}=b_{i,\mbf}$
 for each $i\in\{1,\ldots,n\}$ and for this to hold we must have 
in particular
$a_{\leq n-1}=(a_{0,\mbt},a_{0,\mbf},\ldots,a_{n-1,\mbt},a_{n-1,\mbf})\leq_t (b_{0,\mbt},
b_{0,\mbf},
\ldots,b_{n-1,\mbt},b_{n-1,\mbf})=b_{\leq n-1}$. 
Assume by induction that
 we already know necessary and sufficient conditions for $a_{\leq n-1}\leq_t b_{\leq n-1}$. Now $(a\wedge b)_{n,\mbt}=(a_{n,\mbt}\wedge b_{n,\mbt}) \vee h_{n}((a\wedge b)_{n-1,\mbt})\vee h_{n}((a\wedge b)_{n-1,\mbf})=(a_{n,\mbt}\wedge b_{n,\mbt}) \vee a_{n-1,\mbt}\vee a_{n-1,\mbf}$. And from here we obtain $ a_{n,\mbt}=(a\wedge b)_{n,\mbt}$ if and only if $a_{n-1,\mbf}=1$ or 
$a_{n,\mbt}\leq b_{n,\mbt}$. 
A similar argument applies to  $b_{n,\mbf}$.
From this we can deduce that 
the truth order is that induced by  the lexicographic order  on 
the power 
$(\two\times \two^{\partial})^{n+1}$ of the poset $\two \times \two^\partial$; 
here  $\two^\partial$ denotes the order dual of $\two$, that is $\{0,1\} $ with strict order in which $1 < 0$.  Thus  $\two \times \two^\partial$ is just
$\FOUR$ in its truth order (recall Fig.~\ref{fig:FOUR+SEVEN}(i)).  
 By way of illustration,  Fig.~\ref{fig:K1-orders} 
shows the knowledge and truth orders on $\Seq_n \odot \Seq_n$ for $n=1$. In it 
we have labelled  the elements with 
 four-element binary strings 
$a_{0,\mbt}a_{0,\mbf}a_{1,\mbt}a_{1,\mbf}$.

\begin{figure}[ht]
\begin{center}
\begin{tikzpicture}[scale=.85] 

\li{(1,1)--(2,2)--(3,1)}
\li{(1,1)--(2,0)--(3,1)}
\li{(1,-1)--(2,0)--(3,-1)}
\li{(1,-1)--(2,-2)--(3,-1)}

\epo{2,2}
\epo{1,1}
\epo{3,1}
\epo{2,0}

\epo{1,-1}
\epo{3,-1}
\epo{2,-2}

\node at (2.5,2.2) {$\scriptstyle{1 1 1 1}$};
\node at (0.5,1)   {$\scriptstyle{0 1 1 1}$}; 
\node at (3.5,1)   {$\scriptstyle{1 0 1 1}$}; 
\node at (2,0.5)   {$\scriptstyle{0 0 1 1}$}; 
\node at (0.5,-1)  {$\scriptstyle{0 0 0 1}$}; 
\node at (3.5,-1)  {$\scriptstyle{0 0 1 0}$}; 
\node at (2,-2.3)  {$\scriptstyle{0 0 0 0}$};

\li{(7.5,2)--(7.5,-2)--(11.5,0)--(7.5,2)}
\li{(9.5,1)--(9.5,-1)}

\epo{7.5,2}
\epo{7.5,0}
\epo{7.5,-2}
\epo{9.5,1}
\epo{9.5,0}
\epo{9.5,-1}
\epo{11.5,0}

\node at (7,2.1)   {$\scriptstyle{1 0 1 1}$}; 
\node at (7,0)     {$\scriptstyle{1 1 1 1}$};  
\node at (7,-2.1)  {$\scriptstyle{0 1 1 1}$};
\node at (10,1.1)  {$\scriptstyle{0 0 1 0}$}; 
\node at (8.9,0)   {$\scriptstyle{0 0 1 1}$}; 
\node at (10,-1.1) {$\scriptstyle{0 0 0 1}$}; 
\node at (12,0)    {$\scriptstyle{0 0 0 0}$};
\end{tikzpicture}
\caption{The knowledge order (left) and truth order (right)  of $\Seq_1 \odot \Seq_1$ \label{fig:K1-orders}}
\end{center}
\end{figure}
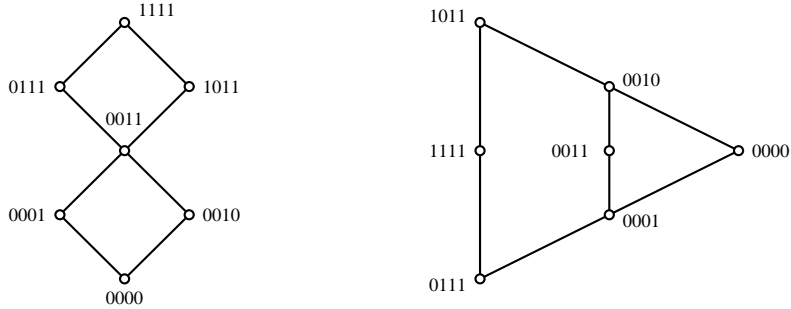 

Our product representation applied to the default sequence $\Seq_n \odot \Seq_n$ associated with the algebra $\K_n$ tells us how  $\K_n$  inherits its knowledge order and its truth order
from a full power.  The latter  is a power of  what is known as the twist structure based on $\two$ 
(see for example \cite{OW} for a discussion of this notion).
We stress that the lexicographic ordering on powers arises only 
because the default sequences we have been considering  are so very  simple. 
(Of course, too, the truth operations are not the join and meet 
inherited from $(\two\times \two^{\partial})^{n+1}$ with the lexicographic order.)
Returning  full circle to our starting point in Section~\ref{intro} we can see
how the particular default bilattices $\mathcal{SEVEN}$, $\mathcal{TEN}$, $\ldots$ relate  
to twist structures in a way which generalises the construction of $\FOUR$
as a twist structure.
\end{ex}


\section{Dualities for  quasivarieties of prioritised default bilattices}
\label{sec:dualitiesKn}

In this 
final section we  turn our attention from the variety  
${\CV_n=\HSP(\K_n)}$ to the quasivariety $\CQ_n =\ISP(\K_n)$, 
for $n\geq1$.  
 We cannot expect there to be a 
product representation 
entirely within the quasivariety because our work in the preceding 
sections heavily involves homomorphic images.  
However $\CQ_n$ does come within the scope of natural duality theory.  
Indeed this class  will have a single-sorted duality based on an 
alter ego $\twiddle{\spc{K}_n}$ 
and from which we can obtain a concrete 
representation for the members of  $\CQ_n$.   
We record below the theorems we can obtain in this way.  
The proofs are somewhat simpler than  those of the corresponding 
results for~$\CV_n$.
 We omit the details, commenting  just on a few salient points.

\begin{thm}\label{thm:dualityKn}
Consider $\CQ_n$.
Then \, 
$\twiddle{\spc{K}_n} 
= ( K_n ; S_{n,n},\ldots,S_{n,0},\Tp )$
yields a strong  {\upshape(}and hence full\,{\upshape)} 
duality on\, $\CQ_n$.  Moreover,  this duality  is optimal. 
\end{thm} 

In outline, the proof of Theorem~\ref{thm:dualityKn} proceeds 
as follows.
Theorem~\ref{thm:Snmonly}, combined with very simple entailment 
arguments,
   leads to the alter ego we present. 
Besides  the entailment constructs of converse and trivial relations 
we need also intersection \cite[Section~2.4]{CD98}.  
The proof of optimality is similar to the first part of the proof of 
Theorem~\ref{thm:optimalstrong}: 
Lemma~\ref{lem:condH} implies that 
$\CQ_n(\S_{n,m},\K_n)$ contains just the restrictions $\rho_1$, $\rho_2$ of the coordinate projections.  
The map from $\CQ_n(\S_{n,m},\K_n)$ to  
$\twiddle{\spc{K}_n}$
sending  $\rho_1$ to $\top_{\!\!m}$ and $\rho_2$ to 
$\mbt_m$ can be shown to 
preserve $S_{n,j}$ for $j\ne m$ but to fail to preserve $S_{n,m}$.  
 
We remark that we do gain here by selecting just  the 
$(n+1)$ relations $S_{n,0}, \dots, S_{n,n}$ rather than the full set of 
binary algebraic relations, of which there are $3n+4$.  
But this gain in simplicity is much less marked than that 
which we saw for the multisorted duality for $\CV_n$.  

We now state a characterisation of the dual category 
$\IScP(\twiddle{\spc{K}_n})$.
We observe first that each of the relations in 
$\twiddle{\spc{K}_n}$, 
except the partial order $S_{m,m}$, is a quasi-order which is not a 
partial order (see Fig.~\ref{fig:Snm} for the case of $n=2$). 
The need for the following  notion should now come as no surprise.
Let $\X = (X;\Tp)$ be a compact space 
and let $\preccurlyeq$ be a quasi-order on~$X$.
We say that $\preccurlyeq$ is a \defn{Priestley quasi-order} 
if $x\not\preccurlyeq y$ implies that there exists a clopen set 
$U$ that is a $\preccurlyeq$-up-set  for which 
 $x\in U$ and $y\notin U$. 
(If in addition $\preccurlyeq$ is a partial order then 
$(X; \preccurlyeq, \Tp)$ is a Priestley space.)
The strategy for the proof of  Theorem~\ref{thm:ISPdualcat} is 
exactly the same as that of Theorem~\ref{Thm:CharDualVn}:  
we exploit
the Separation Theorem for Topological Quasivarieties 
\cite[Theorem~1.4.4]{CD98}.

\begin{thm} \label{thm:ISPdualcat}
Let 
$\X=(X;\preccurlyeq_0,\preccurlyeq_1,\ldots,\preccurlyeq_n,\Tp)$ 
be a structured topological space. Then\, 
$\X\in\IScP(\twiddle{\spc{K}_n})$
if and only if
\begin{newlist}
\item[{\upshape(i)}] $(X;\preccurlyeq_0,\Tp)$ is a Priestley space;
\item[{\upshape(ii)}] $\preccurlyeq_i$ is a Priestley quasi-order 
extending $\preccurlyeq_0$, for $i\in\{1,\ldots,n\}$;
\item[{\upshape(iii)}]$\preccurlyeq_0\, \subseteq\, \preccurlyeq_1\, \subseteq\cdots\subseteq\, \preccurlyeq_n$;
\item[{\upshape(iv)}]
 $\succcurlyeq_i\,\subseteq \,\preccurlyeq_j$, for  $0\leq i<j\leq n$.
\end{newlist}
\end{thm}

\section*{Acknowledgements}  
The first author was supported by a Marie Curie Intra European Fellowship within the 7th European Community Framework Program (ref. 299401-FP7-PEOPLE-2011-IEF). 
The second author was supported by  the Claude Leon Foundation 
during the writing of this paper.  He 
acknowledges also the support of the Rhodes Trust during the period  in which he was studying for his DPhil degree at the University of Oxford, under the supervision of the third author:  during this period he undertook the initial research from 
which the paper subsequently  evolved.

\bibliographystyle{amsplain}

\begin{thebibliography}{99}


\bibitem{BS81}S. Burris,  H.P.  Sankappanavar, A Course in Universal Algebra 
(Millenium Edition).
Available at 
\url{www.math.uwaterloo.ca/~snburris/htdocs/UALG/univ-algebra.pdf}.
		
		
\bibitem{CP1} L.M.  Cabrer, H.A. Priestley, Distributive bilattices from the perspective of natural duality. 
Preprint, 2013 (arxiv: 1308.4495).


\bibitem{CPcop} L.M.  Cabrer, H.A. Priestley,  Coproducts of distributive lattice-based algebras,
Algebra Universalis (To appear, arxiv: 1308.4650).



\bibitem{CD98} D.M. Clark,  B.A.  Davey, Natural Dualities for the Working Algebraist, 
Cambridge University Press, Cambridge, 1998.



\bibitem{BD13}  B.A. Davey, The product representation theorem for interlaced pre-bilattices: some historical remarks,  
Algebra Universalis (2013). Available at \url{dx.doi.org/10.1007/s00012-013-0258-8}

\bibitem{DP87}
B.A. Davey, H.A. Priestley,
Generalized piggyback dualities and applications to Ockham algebras,
Houston J. Math. 13 (1987) 101--117.

\bibitem{ILO2} B.A.  Davey,  H.A. Priestley, Introduction to Lattices
and Order, second  edition,  Cambridge University Press, 2002.

\bibitem{Encheva10} S. Encheva, S. Tumin, Application of default logic in an intelligent tutoring system, 
In: Network-Based Information Systems, LNCS 4658,  486--494, Springer, 2007. 

\bibitem{Pix}  A.L. Foster, A.F.~Pixley,   Semi-categorical algebras II,
Math.  Zeitschr.  85 (1964) 169--184.

\bibitem{Gins86} M.L. Ginsberg, Multi-valued logics, In: Proceedings of the 5th National Conference on Artificial Intelligence, 
243--249, Morgan Kaufman, 1986. 

\bibitem{Gins88} M.L. Ginsberg, Multivalued logics: a uniform approach to inference in Artificial Intelligence, 
Computational Intelligence 4 (1988)  265--316. 

\bibitem{OW}
S. Odintsov, H. Wansing,
Modal logics with Belnapian truth values,
J. Appl. Non-classical Logics 
20  (2010), 279--301.

\bibitem{Rei80}  R. Reiter,  A logic for default reasoning, Artificial Intelligence 
13  (1980) 81--132. 

\bibitem{Sak} C. Sakama, Ordering default theories and nonmonotonic logic programs,  Theoret. Comp. Sci.  338 (2005), 127--152.

\bibitem{Shet2006} V.D. Shet,  D. Harwood,  L.S.  Davis, Multivalued default logic for identity maintenance in visual surveillance,
In: Proceedings of the 9th European conference on Computer Vision, 	Part IV, 119--132, Springer,  
2006. 

\end{thebibliography}
\renewcommand{\bibname}{References}
\raggedright

\end{document}